\documentclass[a4paper,10pt]{article}

\usepackage{latexsym} \usepackage{mathrsfs}

\usepackage{enumerate}
\usepackage{amsmath}\usepackage{dsfont} \usepackage{amssymb}
\usepackage{amsfonts} \usepackage{amsthm} \usepackage{amsxtra}

\usepackage{hyperref}

\usepackage{epsfig} \input{xy} \xyoption{all}
\usepackage{graphics}

\newcommand{\abs}[1]{\left\lvert{#1}\right\rvert}
\newcommand{\norm}[1]{\left\|{#1}\right\|}

\newcommand{\mc}{\mathcal} 

\newcommand{\R}{\mathbb{R}}\newcommand{\N}{\mathbb{N}}
\newcommand{\Z}{\mathbb{Z}}

 \newcommand{\ie}{i.e.\ }
 
\newcommand{\af}{\overleftarrow{f}}
\newcommand{\arr}{\overleftarrow}
\newtheorem{theorem}{Theorem}[section]
\newtheorem{claim}{Claim}[section]

\newtheorem{proposition}[theorem]{Proposition}
\newtheorem{lemma}[theorem]{Lemma}
\newtheorem{corollary}[theorem]{Corollary}

\theoremstyle{definition} 
\newtheorem{example}[theorem]{Example}
\newtheorem{examples}[theorem]{Examples}
\theoremstyle{remark} \newtheorem{remark}[theorem]{Remark}

\newcommand{\Homeo}[1]{\mathrm{Homeo}_{#1}}
\newcommand{\End}[1]{\mathrm{End}^{#1}}
\newcommand{\Mor}{\mathrm{Mor}}

 \DeclareMathOperator{\Lip}{Lip}
\DeclareMathOperator{\Per}{Per}

\begin{document}

\title{Structural stability of the inverse limit  of endomorphisms}

\author{Pierre Berger~\footnote{berger@math.univ-paris13.fr, LAGA
    Université Paris 13} and Alejandro Kocsard~\footnote{akocsard@id.uff.br, Universidade Federal Fluminense}} 
\date{\today}

\maketitle
\begin{abstract}

  We prove that every endomorphism which satisfies Axiom A and the strong transversality conditions is $C^1$-inverse limit
  structurally stable.  These conditions were conjectured to be
  necessary and sufficient.
  This result is applied to the
  study of unfolding of some homoclinic tangencies.
   This also achieves a characterization of $C^1$-inverse limit
  structurally stable covering maps. 
\end{abstract}
\section*{Introduction}

Following Smale \cite{Sm67}, a diffeomorphism $f$ is $C^r$-structurally stable if any $C^r$-perturbation $f'$ of $f$ is conjugate to $f$ via a homeomorphism $h$ of $M$:
\[f\circ h= h\circ f'.\]
A great work was done by many authors to provide a satisfactory
description of $C^1$-structurally stable diffeomorphisms, which starts
with Anosov, Smale, Palis and finishes with Robinson
\cite{Rs} and Ma\~{n}\'e \cite{Mane}. Such diffeomorphisms are those
which satisfy Axiom A and the strong transversality condition.

The descriptions of the structurally stable maps for smoother topologies ($C^r$, $C^1$, holomorphic...)
remain some of the hardest, fundamental and open questions in dynamics.

Hence the description of $C^r$-structurally stable endomorphisms ($C^r$-maps of a manifold not necessarily bijective) with critical points (points at which the differential is not surjective) is even harder.


Indeed, this implies that the critical set must be stable (\ie the map must be equivalent to its perturbations via homeomorphisms) and so that $r$ must be at least $2$.  We recall that the description of critical sets which are stable is still an open problem \cite{Ma70}. 

It is not the case when we consider the structural stability of the inverse
limit. 
We recall that the \emph{inverse limit set} of a $C^1$-endomorphism $f$ is the space of the full orbits
$(x_i)_i \in M^\mathbb Z$ of $f$. The dynamics induced by $f$ on its
inverse limit set is the shift. The endomorphism $f$ is
\emph{$C^1$-inverse limit stable} (or equivalently inverse
limit of $f$ is $C^1$-structurally stable) if for every $C^1$ perturbation $f'$
of $f$, the inverse limit set of $f'$ is homeomorphic to the one of
$f$ via a homeomorphism which conjugates both induced dynamics and
which is $C^0$-close to the canonical inclusion into $M^\mathbb Z$.

When the dynamics $f$ is a diffeomorphism, the inverse limit set $\arr
M_f$ is homeomorphic to the manifold $M$. The $C^1$-inverse limit
stability of $f$ is then equivalent to the $C^1$-\emph{structural
  stability} of $f$: every $C^1$-perturbation of $f$ is conjugated to
$f$ via a homeomorphism of $M$ $C^0$-close to the identity.

The concept of inverse limit stability is an area of great interest for semi-flows given by PDEs,
although still at its infancy \cite{Quandt, Joly2010generic}.

There were many works giving sufficient conditions for an endomorphism to be structurally stable \cite{MP, Przy, BR12}.
The latter work generalized Axiom A and the strong transversality condition to differentiable endomorphisms of
manifolds, and conjectured these conditions to be equivalent to
$C^1$-inverse limit stability. A main point of this work was to give evidences that the notion of inverse stability should be independent to the nature of the critical set (stable or not for instance).  
A similar conjecture was sketched in \cite{Quandt-ano}.

We prove here one direction of this conjecture, generalizing \cite{MP, Przy, BR12, Ri, Rs} :
 \begin{theorem}[Main result] \label{thm:main}
  Every $C^1$-endomorphism of a compact manifold which satisfies Axiom A and the strong transversality condition is $C^1$-inverse
  limit structurally stable.
\end{theorem}
The definitions of Axiom A and the strong transversality condition will be recalled in \textsection \ref{sec:axiom-A-endomorphisms}.

 Joint with the works of
\cite{AMS} and \cite{BR12}, this proves that $C^1$-inverse limit
stable covering maps of manifolds are exactly the $C^1$ covering  maps which
satisfy Axiom A and strong transversality conditions (see \textsection \ref{covering}).  

On the other hand, our main result applies to the dynamical studies of homoclinic
tangencies unfolding as seen in section (see \textsection \ref{unfolding}).

 The proof of the main result is done by generalizing Robbin-Robinson proof of the structural stability with two new difficulties. We will have to handle the geometrical and analytical part of the argument on the inverse limit space which is in general not a manifold as it is the case for diffeomorphisms (see \textsection \ref{analysis}). Also we will have to take care of the critical set in the plane fields constructions and in the inverse of the operator considered (see \textsection \ref{propprincipal}, \ref{princis} and \ref{princiu}).

\thanks{ This work has been Partially supported by the Balzan Research
  Project of J. Palis. We are grateful to A. Rovella
  for helpful discussions.}

\section{Notations and definitions}
\label{sec:notations}

Along this article $M$ will denote a smooth Riemannian compact
manifold without boundary. The distance on $M$ induced by the
Riemannian structure will be simply denoted by $d$. For any $r\in\N$,
we denote by $\End{r}(M)$ the space of $C^r$ endomorphisms of $M$. By
\emph{$C^r$ endomorphism of $M$}, we mean a $C^r$ map $f$ of $M$ into
$M$, which is possibly non surjective and can have a non-empty
\emph{critical set}:
\[C_f:= \{x\in M:\; T_xf \; \text{not surjective}\}.\] We endow
$\End{r}(M)$ with the topology of uniform convergence of the first $r$
derivatives.

Given any $f\in\End{r}(M)$, a subset $\Lambda\subset M$ is
\emph{forward invariant} whenever $f(\Lambda)\subset\Lambda$, and
\emph{totally invariant} when $f^{-1}(\Lambda)=\Lambda$. Note that
totally invariance implies forward invariance.

The set of periodic points of $f$ is denoted by $\Per(f)$ and we write
$\Omega(f)$ for the set of non-wandering points. Observe that
$f(\Per(f))=\Per(f)$ and $f(\Omega(f))=\Omega(f)$, but in general they
are not totally invariant.

Now, let $K$ be a compact metric space and $E\to
K$ a finite dimensional vector bundle over $K$. If $F\subset E$ is a
sub-vector-bundle of $E\to K$, we denote by $E/F$ the quotient
bundle. Note that any (Riemannian) norm $\norm{\cdot}_E$ on $E$
naturally induces a (Riemannian) norm on $E/F$ defining
\begin{displaymath}
  \norm{v_x+F_x}_{E/F}:=\inf_{w_x\in F_x}\norm{v_x+w_x}_E,
  \quad\forall x\in K,\ \forall v_x\in E_x. 
\end{displaymath}

On the other hand, observe that any bundle map $T\colon E\to E$ that
leaves invariant $F$ (\ie $F$ is forward invariant for $T$) naturally
induces a bundle map $[T]\colon E/F\to E/F$.

\subsection{Inverse limits}
\label{sec:inverse-limits}

Given any set $X$ and an arbitrary map $f\colon X\to X$, we define its
\emph{global attractor} by $X_f:=\bigcap_{n\geq 1} f^n(X)$ and its
\emph{inverse limit} by
\begin{equation}
  \label{eq:inv-lim-def}
  \arr{X}_f :=\left\{ \underline{x}=(x_n)_n\in X^{\Z} :
    f(x_n)=x_{n+1},\ \forall n\in\Z\right\}.
\end{equation}

Observe that $f(X_f)=X_f$, but in general $X_f$ is not totally
invariant ($X_f\not= f^{-1}(X_f)$), and $\arr{X}_f\subset
(X_f)^\Z$. Moreover, $f$ acts coordinate-wise on $\arr{X}_f$. In fact,
we can define $\arr{f}\colon \arr{X}_f\to\arr{X}_f$ by
$\af(\underline{x}):=(f(x_n))_n=(x_{n+1})_n$, and in this way, $\af$
turns out to be a bijection and $f$ a factor of it. Indeed, for every
$j\in\Z$ we can define the $j^{\mathrm{th}}$-projection
$\pi_j\colon\overleftarrow{X}_f\to X_f$ by $\pi_j(\underline{x})=x_j$
and then we have
\begin{displaymath}
  \pi_{j+1}=\pi_j\circ \overleftarrow{f}=f\circ\pi_j.
\end{displaymath}
Whenever $X$ is a topological space and $f$ is continuous, we shall
consider $X^\Z$ endowed with the product topology. In this case,
$\arr{X}_f$ turns out to be closed in $X^\Z$ and $\arr{f}$ a
homeomorphism. Of course, $\Per(f)$ and $\Omega(f)$ are contained in
$X_f$ and $\arr{X}_f$ is compact whenever $X_f$ is compact itself.

Finally, when $X$ is endowed with a finite distance $d$, we shall
consider $X^\Z$ equipped with the distance $d_1$ given by
\begin{equation}
  \label{eq:def-d-Z-prod}
  d_1(\underbar{x},\underbar{y}):=\sum_{n\in\Z}
  \frac{d(x_n,y_n)}{2^{\abs{n}}}. 
\end{equation}
The metric space $(X^\Z,d_1)$ is compact if and only if $X$ is
compact itself.

\subsection{Structural and inverse limit stability}
\label{sec:structural-stability}

Two endomorphisms $f,g\in\End{r}(M)$ are \emph{conjugate} when there
exists a homeomorphism $h\in\Homeo{}(M)$ satisfying $h\circ f=g\circ
h$. More generally, the endomorphisms $f$ and $g$ are \emph{inverse
  limit conjugate} whenever there exists a homeomorphism $H\colon
\arr{M}_f\to\arr{M}_g$ such that $H\circ \af=\arr{g}\circ H$. Remark
that the conjugacy relation implies the inverse limit
conjugacy one.

A $C^r$-endomorphism $f$ is \emph{$C^s$-structurally stable} (with
$0\leq s\leq r$) when there exists a $C^s$-neighborhood $\mc{U}$ of
$f$ such that every $g\in\mc{U}$ is conjugate to $f$. Analogously, $f$
is \emph{$C^s$-inverse limit stable} when every $g\in\mc{U}$ is
inverse limit conjugate to $f$.

\subsection{Axiom A endomorphisms}
\label{sec:axiom-A-endomorphisms}

Let $f\in\End{1}(M)$ and let $\Lambda\subset M$ be a compact forward
invariant set. The set $\Lambda$ is \emph{hyperbolic}
whenever there exists a continuous sub-bundle $E^s\subset T_\Lambda M$
satisfying the following properties:
\begin{enumerate}
\item $E^s$ is forward invariant by $Tf$, \ie
  \begin{displaymath}
    T_xf(E^s_x)\subset E^s_{f(x)}, \quad\forall x\in\Lambda;
  \end{displaymath}
\item the induced linear map $[T_xf]\colon T_x M/E^s\to T_{f(x)}
  M/E^s$ is an isomorphism, for every $x\in\Lambda$; (see
  \S\ref{sec:notations} for notation of quotient bundles and induced
  maps)
\item $\forall x\in\Lambda$, $\norm{T_xf\big|_{E^s_x}}<1$ and
  $\norm{[T_xf]^{-1}}<1$, where the first operator norm is induced by
  the Riemannian structure of $M$, and the second by its quotient.
\end{enumerate}

\begin{remark}
  \label{rem:Eu-abstract}
  Notice that despite $E^s$ is contained in $T_\Lambda M$, in general
  we cannot define $E^u$ as a sub-bundle of the tangent bundle.
\end{remark}
 
However, using a classical cone field argument, we show:
\begin{proposition}
  \label{prop:Eu-concrete}
  There exists a continuous family $(E_{\underline x}^u)_{\underline
    x\in \arr{\Lambda}_f}$ of subspaces of $T_\Lambda M$ such that:
  \begin{enumerate}
  \item for every $\underline x\in\arr{\Lambda}_f$, $E_{\underline
      x}^u\subset T_{\pi_0(\underline x)}M$ and $Tf(E_{\underline
      x}^u) = E_{\af(\underline x)}^u$,
  \item for every $\underline x\in\arr\Lambda_f$, the restriction
    $Tf\colon E_{\underline x}^u\to E_{\arr f(\underline x)}^u$ is
    invertible and
    \begin{displaymath}
      \norm{(Tf\big|_{E^u_{\underline x}})^{-1}}<1.
    \end{displaymath}
  \end{enumerate}
\end{proposition}

Given any (small) $\varepsilon>0$ and $\underbar{x}\in
\overleftarrow{\Lambda}_f$, we define the \emph{$\varepsilon$-local
  stable set} of $\underbar{x}$ by
\begin{displaymath}
  W^s_\varepsilon(\underbar{x},f):=
  \Big\{\underbar y\in\overleftarrow{M}_f : \varepsilon \geq
  d_1(\arr{f}^n(\underbar x),\arr{f}^n(\underbar y))
  \xrightarrow{n\to+\infty} 0,\ \forall n\geq 0 \Big\}.
\end{displaymath}
where $d_1$ denotes the distance given by \eqref{eq:def-d-Z-prod}; and
the \emph{$\varepsilon$-local unstable} set of $\underbar{x}$ is
defined analogously by
\begin{displaymath}
  W^u_\varepsilon(\underbar{x},f):=
  \Big\{\underbar y\in\overleftarrow{M}_f : \varepsilon \geq
  d_1\big(\overleftarrow{f}^{-n}(\underbar x),\overleftarrow{f}^{-n}(\underbar y)\big)
  \xrightarrow{n\to+\infty} 0,\ \forall n\geq 0 \Big\}.
\end{displaymath}
The geometry of these sets was described in \cite{BR12}. Let us recall
that $\pi_0\big( W^u_\varepsilon(\underbar{x},f)\big)$ and $\pi_0\big(
W^s_\varepsilon(\underbar{x},f)\big)$ are submanifolds of $M$ (for
$\epsilon$ sufficienlty small).

The endomorphism $f$ satisfies \emph{Axiom A} when $\Omega(f)$ is
hyperbolic and coincides with the closure of $\Per(f)$.

An Axiom A endomorphism satisfies the \emph{strong transversality
  condition} if for every $\underline x,\underline y\in
\Omega(\overleftarrow{f})$ and every $n\ge 0$, the map
$f^n\big|_{\pi_0(W^u_\epsilon (\underline x))}$ is transverse to
$\pi_0(W^s_\epsilon (\underline y))$. This means that for every $z\in
\pi_0(W^u_\epsilon (\underline x))\cap f^{-n}(\pi_0(W^s_\epsilon
(\underline y))$ the following holds:
\begin{displaymath}
  Tf^n \big(T_z \pi_0(W^u_\epsilon (\underline x))\big) + T_{f^n(z)} 
  \pi_0(W^s_\epsilon (\underline y))=T_{f^n(z)} M
\end{displaymath}

An endomorphism which satisfies Axiom A and the strong transversality
condition is called an \emph{$AS$-endomorphism}.

This notion generalizes the one of diffeomorphism. Let us recall some other examples.
\begin{examples}
\begin{itemize}
\item The action of any linear matrix in $\mathcal M_n(\Z)$ in the $n$-dimensional torus is hyperbolic and so  is AS. It is not structurally stable whenever the matrix is not in $SL_n(\Z)$ nor expanding \cite{Prz76}.
\item The constant map $\R^n\ni x\mapsto 0\in \R^n$ satisfies Axiom A and the strong transversality  condition.  
\item The map $\R^n\ni  x\mapsto x^2+c$ satisfies Axiom A and the strong transversality  condition whenever $c$ is such that a (possibly super) attracting periodic orbit exists.
\end{itemize}
\end{examples}
\begin{remark}
Let us notice that if two  endomorphisms $f_1\in C^1(M_1, M_1)$ and $f_2\in C^1(M_2, M_2)$ satisfy Axiom A and the strong transversality condition, then the product dynamics $f_1\times f_2 \in  C^1(M_1\times M_2, M_1\times M_2)$ also satisfy  Axiom A and the strong transversality condition. 

As an endomorphism is AS iff its one of its iterates is AS, it follows that the following delay dynamics is AS if $f\in C^1(M,M)$ is AS:
\[ M^n \ni (x_i)_i\mapsto  (f(x_m),x_1, \dots ,x_{m-1},0,\dots, 0)\in M^n.\]
\end{remark}

From the latter example and remark, we have the following.
\begin{example}\label{delaymap}
For every $c\in \R$ such that $x^2+c$ has an attracting periodic orbit, the following map is AS:
\[(x_i)_i\in \R^n \mapsto  (x_m^2+c,x_1, \dots ,x_{m-1}, 0, \dots ,0)\in \R^n.\]  
\end{example}
We will see that this example appears in the unfolding of generics homoclinic tangency in \textsection \ref{unfolding}.
\section{Applications of main Theorem \ref{thm:main}}
\label{sec:main-result}
%

\subsection{Description of \texorpdfstring{$C^1$}{C1}-inverse limit stable covering maps}\label{covering}
A \emph{$C^1$-covering map} of a compact, connected manifold $M$ is a surjective
$C^1$ endomorphism $f$ of $M$ without critical points.  Then, every
point of $M$ has the same number $p$ of preimages under $f$. We remark
that every distinguish neighborhood $U\subset M$ has its preimage
$\pi^{-1}_0(U)$ in the inverse limit $\arr M_f$ which is homeomorphic
to $U\times \mc{Z}_p$ where $\mc{Z}_p$ is a Cantor set labeling the
different $f$-preorbits of $U$. These homeomorphisms endow $\arr M_f$
with a structure of lamination called the Sullivan solenoid \cite{Su93}.

It follows immediately from a theorem due to Aoki, Moriyasu and Sumi
\cite{AMS} that: if an endomorphism $f$ is
$C^1$-inverse limit stable and has no critical point in the
non-wandering set, then $f$ satisfies Axiom A. By Theorem 2.4 of
\cite{BR12}, if $f$ is $C^1$-inverse limit stable and satisfies Axiom
A, then $f$ satisfies the strong transversality condition. Together
with  Main Theorem \ref{thm:main}, it comes the following description of $C^1$-inverse stable covering  maps.

\begin{theorem}
  \label{coro}
  A $C^1$-covering  map of a compact manifold is $C^1$-inverse limit stable
  if and only if it is an AS-endomorphism.
\end{theorem}
 
\subsection{Application to dynamical study of unfolding homoclinic tangencies}\label{unfolding}
Let $M$ be a manifold of dimension $m$ and let $(f_\mu)_\mu$ be a smooth family of diffeomorphisms of $M$ which has a hyperbolic fixed point $p$ with unstable and stable directions of dimensions $u\ge 1$ and $s\ge 1$ respectively. Hence $m=u+s$.
The following Theorem has been proven in the general case as in \cite{Mora} (Prop. 1). For more restricted cases see \cite{PT93} when $(s,u)=(1,1)$ and Th. 1 \cite{Tatjer} when $(s,u)=(1,2)$.

\begin{theorem}[L. Mora]\label{Mora}
There exist $h\ge u$, an open set of families $(f_{\mu})_{\mu\in \R^h}$ of smooth diffeomorphisms of $M$, which exhibit  at $\mu_0\in \R^h$ an unfolding of a homoclinic tangency at $q\in W^s(p)\cap W^u(p)$, such that there exists a small neighborhood $N_q\subset M$ of $q$, there exists  a small neighborhood  $N_\mu\subset \R^h$ of $\mu_0$ covered by submanifolds $\mathcal L$ of dimension $u$,  satisfying for every $n$ large:
\begin{itemize}
\item $\mu_0$ belongs to every submanifold $\mathcal L$ and the intersection of two different such manifolds $\mathcal L$ is the single point $\mu_0$,
\item for every $\mathcal L$, there is parametrization $\gamma_n$ of $\mathcal L$  by $\R^u$, such that for every $\mu=\gamma_n(\underline b)\in \mathcal L\setminus \{\mu_0\}$, there is a chart $\phi_\mu$ of $N_q$, such that the rescaled first return map has the form:
\[\phi_\mu \circ f^n_{\gamma(\underline b)} \circ \phi_\mu^{-1}\colon \R^u\times \R^s\longrightarrow \R^u\times \R^s=\R^m\]
\[(\underline x,\underline y)\mapsto (x_u^2+b_u +\sum_{i=1}^{u-1} b_i x_i, x_1,\dots , x_{u-1},0,\dots ,0)+E_\mu(\underline x, \underline y),\]
with $\underline b=b_1,\dots b_u$, $\underline x=(x_1,\dots , x_u)$, $\underline y = (y_1,\dots, y_s)$ and $E_\mu\in C^\infty(\R^m,\R^m)$ small in the compact-open $C^r$-topology for every $r$ when $n$ is large.\end{itemize}
\end{theorem}

In particular, near the curve $\{\mu_n(a):= \gamma_n(0,\dots, 0,a), \; a\in \mathbb R\}$, the rescaled first return map $f_{\mu_n(a)}$ is $C^r$ close to the endomorphism:
\[F_a:=(x_1,\dots , x_u,y_1,\dots, y_s)\mapsto (x_u^2+a, x_1,\dots , x_{u-1},0,\dots ,0).\]

For an open and dense set of parameters $a$, the map $x\mapsto x^2+ a$ has an attracting periodic orbit from \cite{Lyufatou, GrSw}. Then its non-wandering set is the union of an attracting periodic orbit with an expanding compact set. By example \ref{delaymap}, we know that $F_a$ is AS. Moreover we can extend $F_a$ to the $n$-torus which is the product of $n$-times the one point compactification of $\mathbb R$. Its extension is analytic and $AS$.

Hence by Theorem \ref{thm:main}, the inverse limit of $F_a$ restricted to bounded orbits is conjugate to an invariant compact set of 
$f_{\mu_n(a)}^n|U_n$ with $n$ large.

In particular, if such an $a$ is fixed and then $n$ is taken large, then there exist an open set $V_n$ of $M$ and a neighborhood $W_n$ of $\mu_n (a)$ such that for every $a'\in W_n$, $f_{a'}^{nu}|V_n$ has its maximal invariant compact set conjugated to the product of $u$-times the  inverse limit dynamics of $x^2+a$ (restricted to the bounded orbits).     

For instance, when $a=0$, then the non-wandering set of $x\mapsto x^2$ consists of the attracting fixed point $0$ and the repelling fixed point $1$.  On the other hand the non-wandering set of $F_0$ is $\{0,1\}^u\times \{0\}$. We remark also that the set of points for which the orbit is bounded is homeomorphic to the square $[0,1]^u\times \{0\}$ via the first coordinate projection. Hence for $a$ small, the maximal invariant of $F_a$ is a topological $u$-cube bounded by the stable and unstable manifolds of the hyperbolic continuation of the non-wandering points.

This example would be more interesting for a parameter $a$ with positive entropy, but it present already most of the difficulties for the geometrical part of Theorem \ref{thm:main} proof (but the fact that the geometry of inverse limit space is much simpler than in the positive entropy case for instance). We will keep in mind this example.

\section{Proof of Main Theorem \ref{thm:main}}
\subsection{Sufficient conditions for the existence of a conjugacy} 
We want to find, for every $g$ which is $C^1$-close to $f$, 
 a continuous map $h\colon \arr M_f\to \arr
M_g$ which is close to the canonical injection $\arr M_f
\hookrightarrow M^\mathbb Z$ and satisfies $h\circ \af = \arr g\circ
h$. This is equivalent to find a continuous map $h_0: \arr M_f \rightarrow M$ satisfying
\begin{equation}\tag{$\mathring C_1$}
  \label{eq:h0-def}
  h_0\circ \af = g\circ h_0,
\end{equation}
and which is \emph{$C^0$-close} to the zeroth coordinate projection
$\pi_0$. This means that for every $\eta>0$ small and every $g$
sufficiently $C^1$-close to $f$, $h_0$ satisfies
\begin{equation}\tag{$\mathring C_2$}
  \label{h0-small}
  \sup_{\underline x\in \arr M_f} d(h_0(\underline x),
  \pi_0(\underline x))\le \eta.
\end{equation}

Indeed, one can construct such an $h_0$ from such an $h$ and
\emph{vice versa} writing:
\begin{displaymath}
  h_0:= \pi_0\circ h, \quad h:= ( h_0\circ \af ^n)_{n\in \mathbb Z}.
\end{displaymath}

Let us suppose the existence of such an $h$. We would like $h$ to be a
homeomorphism, so let us find sufficient conditions to ensure its
injectiveness.

In the Anosov case this follows easily from (\ref{eq:h0-def}) and
(\ref{h0-small}). In fact, if two points $\underline x$ and
$\underline y$ have the same image by $h$, then these two $f$-orbits
must be uniformly close by (\ref{eq:h0-def}) and (\ref{h0-small}), and
so they are equal by expansiveness. In the wider case of AS dynamical systems,
we shall consider the following Robbin metric on $\arr M_f$:
\begin{displaymath}
  d_\infty (\underline x, \underline y)= \sup_{i\in \mathbb Z} d(x_i,y_i).
\end{displaymath}
For every $\underline x, \underline y \in \arr M_f$, let us observe
that:
\begin{equation}
  \label{rem:semiequivalencedemetric}
  d_1(\underline x,\underline y)= \sum_i \frac{d(x_i,y_i)}{2^{|i|}}\le
  \sum_i \frac{d_\infty (\underline x,\underline y)}{2^{|i|}}= 3
  d_\infty (\underline x,\underline y).
\end{equation}
This metric enabled Robbin \cite{Ri} to find a sufficient condition on
$h$ to guarantee its injectiveness. We adapt it to our context.

Proposition 4.14 of \cite{BR12} gives a geometric interpretation of
the metric $d_\infty$. After showing that $\arr M_f$ is a finite union
of laminations, the leaves of which are intersection of stable sets
with unstable manifolds of points in $\arr \Omega_f$, we proved that
$d_\infty$-distance between any two of these leaves is positive. Moreover the restriction of $d_\infty$ to each leaf is equivalent to a Riemannian
metric on its manifold structure.

Choosing $\eta>0$ small, any continuous map $h_0\colon \arr M_f\to M$ satisfying
(\ref{h0-small}) can be written as a perturbation of
$\pi_0$ via the exponential map $\exp\colon TM\to M$ associated to the
Riemannian metric of $M$.  For the sake of simplicity, let us fix a
bundle trivialization $TM\subset M\times \R^N$, for some positive
integer $N$. As $h_0$ satisfies (\ref{h0-small}) (with $\eta$ small),
there exists $w\colon \arr M_f\to \R^N$ such that $(\pi_0(\underline
x), w(\underline x))\in TM\subset M\times \R^N$ and
\begin{equation}
  \label{eq:w-def-of-h0}
  h_0(\underline x)=\exp_{\pi_0(\underline x)}(w(\underline x)).
\end{equation}
%
%
Now let us extend the Riemannian metric
$(\langle\cdot,\cdot\rangle_x)_{x\in M}$ of $M$ to an Euclidean norm
$(\|\cdot \|_x)_{x\in M}$ on the bundle $M\times \R^N\rightarrow M$.
Let us denote by $\Lambda w\in [0,\infty]$ the $d_\infty$-Lipschitz
constant of $w$, \ie
\begin{equation}\label{deflip}
  \Lambda w:=\sup_{\underline x, \underline x'\in \arr M_f}
  \frac{\|w(\underline x)-w(\underline x')\|_{\pi_0(\underline
      x)}}{d_\infty (\underline x, \underline x')}. 
\end{equation}
Here is the Robbin condition:
\begin{equation}
  \tag{$C_3$}
  \label{h0-lip}
  \Lambda w\le  \eta.
\end{equation}

\begin{proposition}[Robbin~\cite{Ri}]\label{Robbin}
  There exists $\eta>0$ which depends only on the Riemannian metric of
  $M$, such that for every pair $f$ and $g$ of $C^1$-endomorphisms of
  $M$, if there exists $h\colon \arr M_f\rightarrow \arr M_g$
  satisfying (\ref{eq:h0-def}), (\ref{h0-small}) and (\ref{h0-lip}),
  then $h$ is injective.
\end{proposition}

\begin{proof} 
  Let $\underline x, \underline x'\in \arr M_f$ be such that
  $h(\underline x)=h(\underline x')$. Note that by (\ref{eq:h0-def})
  and (\ref{h0-small}), the point $\pi_i\circ h(\underline x)$ is
  $\eta$-close to $\pi_i(\underline x)$, for every $i$. Thus
  $\pi_i(\underline x)$ and $\pi_i(\underline x')$ are $2\eta$-close
  for every $i\in \mathbb Z$.

  Let $i\in \mathbb Z$ be such that $d_\infty(\underline x ,
  \underline x')\le 2d(x_i, x'_i)$. We recall that $\pi_i\circ
  h(\underline x)=h_0\circ \arr f^i(\underline x)=\exp_{x_i}(w\circ
  \arr f^i(\underline x))$, and so:
  \begin{displaymath}
    \exp_{x_i} (w\circ \arr f^i(\underline x))=\exp_{x'_i} (w\circ
    \arr f^i(\underline x')) 
  \end{displaymath}
  The exponential maps $\exp_{x_i}$ and $\exp_{x'_i}$ produce two
  charts centered at $x_i$ and $x_i'$, and modeled on the vector
  subspaces $T_{x_i} M$ and $T_{x'_i} M$ of $\mathbb R^N$. The
  coordinates change of these charts is the translation by the vector
  $\exp_{x_i}^{-1} x_i'$ 
  plus a linear map $L$ bounded by a constant $K$ times $d(x_i,x'_i)$, where $K$ depends only on the curvature of $M$.
    Thus, in $\mathbb R^N$, it holds:
  \begin{equation}
    \label{equadinj}
    \exp_{x_i}^{-1} x_i'+(id + L)\circ w\circ \arr f^i(\underline x)= w\circ \arr
    f^i(\underline x')+o(d(x_i,x_i')),
  \end{equation}
  We recall that $d_\infty(\underline x, \underline x')\le 2
  d(x_i,x_i')\le 4\eta$ is small.

  On the other hand by (\ref{h0-lip}):
  \begin{displaymath}
    \|w\circ \arr f^i(\underline x)-w\circ \arr f^i(\underline
    x')\|_{x_i}\le \eta d_\infty (\underline x, \underline x')
  \end{displaymath}

  Thus, replacing each term of equality (\ref{equadinj}) by these
  estimates, it holds:
  \begin{displaymath}
    \begin{split}
      \frac{d_\infty(\underline x, \underline x')}{2} &\le d(x_i,x_i')=\|\exp_{x_i}^{-1} x_i'\| \\
      &\le \|w\circ \arr f^i(\underline x)-w\circ \arr f^i(\underline x')\|_{x_i}+\|L\|\eta+o(d(x_i,x_i'))\\
      &\le \eta d_\infty (\underline x, \underline x')+Kd_\infty (\underline x, \underline x') \eta+ o(d_\infty
      (\underline x, \underline x'))
    \end{split}
  \end{displaymath}
  This implies $d_\infty (\underline x, \underline x')=0$ and so
  $\underline x=\underline x'$.
\end{proof}

On the other hand, in Proposition 5.4 of \cite{BR12} it is showed the
following:
\begin{proposition}\label{surjdeconj}
  For every $AS$ $C^1$-endomorphism $f$ of $M$, there exists $\eta>0$
  such that for every endomorphism $g$ sufficiently close to $f$, if
  there exists a $d_1$-continuous and injective $h:\; \arr
  M_f\rightarrow \arr M_g$ satisfying (\ref{eq:h0-def}) and
  (\ref{h0-small}) with $f$, $g$ and $\eta$, then $h$ is surjective
  onto $\arr M_g$.
\end{proposition}

Hence if we prove that for every $g$ $C^1$-close to an AS endomorphism
$f$, there exists a continuous map $h_0$ satisfying (\ref{eq:h0-def}),
(\ref{h0-small}) and (\ref{h0-lip}), then Propositions \ref{Robbin}
and \ref{surjdeconj} imply that $g$ is inverse limit conjugate to $f$,
and so that $f$ is $C^1$-inverse limit stable. In other words, to prove Theorem
\ref{thm:main} it remains only to prove the following:
\begin{proposition}
  \label{part1}
  Let $f$ be a $C^1$-AS endomorphism. For every $\eta>0$ and for every
  endomorphism $g$ sufficiently $C^1$-close $f$, there exists a
  continuous map $h_0\colon \arr M_f\rightarrow M$ satisfying
  (\ref{eq:h0-def}), (\ref{h0-small}) and (\ref{h0-lip}) with $f$, $g$
  and $\eta$.
\end{proposition}
Therefore the remaining part of this manuscript is devoted to the
proof of this proposition, by using the contraction mapping Theorem.

\subsection{A contracting map on a functional space}
\label{sec:add-not}



Let $\Gamma$ be the space of functions $w:\; \arr M_f\to \R^N$ which
are continuous for $d_1$ and $d_\infty$-Lipschitz (\ie they satisfy $
\Lambda(w)<\infty$, where $ \Lambda(w)$ is defined as in
\eqref{deflip}).
We endow $\Gamma$ with the uniform norm:
\begin{displaymath}
  \|v\|_{C^0}:=\max_{\underline x \in \arr M_f} \|v(\underline x)
  \|_{\underline x}.
\end{displaymath}

\paragraph{Equivalent conditions in the space $\Gamma$:}
We recall that $M\times \R^N\supset TM$ is a trivialization. Moreover
we have already fixed an Euclidean structure $(\|\cdot \|_x)_{x\in M}$
on $M\times\R^N$ which extends the Riemannian metric
$(\langle\cdot,\cdot\rangle_x)_{x\in M}$ on $TM$. Let $p_x\colon
\mathbb R^N\to T_xM$ be the orthogonal projection given by
$\|\cdot\|_x$.

Any $g$ sufficiently  $C^0$-close to $f$ 
induces the following map from a neighborhood $N_\Gamma$ of $0\in
\Gamma$ into $\Gamma$:
\begin{displaymath}
  \Phi_f^g(w):= \underline x \mapsto \exp^{-1}_{x_0}
  \Big(g\circ \exp_{x_{-1}}\big( p_{x_{-1}} \circ w\circ \arr
  f^{-1}(\underline x) \big)\Big), \quad \forall w\in N_\Gamma,
\end{displaymath}
where $x_i=\pi_{i}(\underline x)$ for every $i$, as defined in
\S~\ref{sec:inverse-limits}.

For every $\eta$ small, and for every $g$ sufficiently close to $f$,
to find $h_0$ satisfying conditions \eqref{eq:h0-def},
(\ref{h0-small}) and (\ref{h0-lip}) is equivalent to find $w\in
\Gamma$ satisfying
\begin{align}
  \tag{$C_1$}
  \Phi_f^g(w)=w\\
  \tag{$C_2$} \|w\|_{C^0}
  \le \eta\\
  \tag{$C_3$} \Lambda(w)\le \eta
\end{align}
Indeed, by $(C_1)$ any such $w$ satisfies $w(\underline x)\in T_{x_0}
M$, for every $\underline x\in \arr M_f$. It is then easy to remark that
$h_0:\underline x\mapsto \exp_{x_0} ( w(\underline x))$ satisfies
$(\mathring C_1)$ and $(\mathring C_2)$. 

\paragraph{Strategy:} To solve this (implicit) problem, let us regard
the partial derivative of $\Phi_f^f$ at $0\in \Gamma$ with respect to
$w\in \Gamma$:
\begin{displaymath}
  D_0\Phi_f^f = w\mapsto \big[\underline x\mapsto  T_{x_{-1}}f(
  p_{x_{-1}}\circ w\circ \af^{-1}(\underline x))\big]. 
\end{displaymath}

The first difficulty that appears is the following: if $f$ is only
$C^1$, in general the map $D\Phi_f^f$ does not leave invariant the
space $\Gamma$. In the $C^2$-case, Robbin's strategy in \cite{Ri} consists in
solving ($C_1$)-($C_2$)-($C_3$) by finding a right inverse for
$D\Phi_f^f-id$, and then by following a classical proof of the
implicit function theorem which uses the contraction mapping Theorem.

A second difficulty which will appear is that $Tf$ is possibly non-invertible, and this will give us manyd ifficulties to construct this right inverse with bounded norm. To eliminate some of these, in Lemma \ref{Finversible}, we will suppose $N$  twice larger than necessary to embed $TM$ into $M\times \R^N$.

\paragraph{Robinson trick.}

%

If $f$ is not $C^2$ but only $C^1$, then there is a continuous family of
$C^0$-maps  $(x\mapsto F^\delta_x )_\delta$ from $M$ in the space of linear maps of $\R^N$, such that $F^0_x=D_xf\circ p_x$ and each $F^\delta$ is $C^\infty$, for $\delta>0$.

 Such a family is easily constructed by smoothing $f$ to a map $f_\delta$ (by using the classical technique of convolutions with
mollifier functions on $f$), and then looking at its differential.

Let us regard the following linear bundle morphism $F^\delta$:
\begin{equation}
  \label{eq:F-definition}
  \begin{array}{cccc}  
    F^\delta\colon& \arr M_f\times \mathbb R^N&\to& \arr
    M_f\times \mathbb R^N\\ 
    &(\underline x, v)&\mapsto &(\arr f(\underline
    x),F^\delta_{x_0}(v))
  \end{array}
\end{equation}  
Note that $F^\delta$ is still over $\af$,
\ie the following diagram commutes:
\begin{displaymath}
  \xymatrix{
    \arr M_f\times \R^N\ar[d]\ar[r]^{F^\delta} & \arr M_f\times \R^N\ar[d] \\
    \arr M_f\ar[r]^{\af} & \arr M_f \\
  }
\end{displaymath}

\begin{lemma}\label{Finversible} If $N$ is large enough, then we can suppose moreover that $F^\delta_x$ is invertible for every $\delta>0$ and $x\in M$.
\end{lemma}
\begin{proof} 
Put $N'=2N$. Let $f_\delta$ be a smooth endomorphism $C^1$-close $f$ when $\delta$ is small. 
We extend the projection $p_x\colon \R^N\to T_x M$ to $\R^{N'}= \mathbb R^N\times \mathbb R^N$ by 
\[p_x\colon \mathbb R^N\times \mathbb R^N\ni (v_1,v_2)\mapsto p_x(v_1)\in \mathbb R^N.\] Also we identify $\R^N$ to $\R^N\times \{0\}\subset \R^{N'}$. Let us regard:
\[F^\delta_ x\colon \R^{N'}\ni v= (v_1,v_2)\mapsto (T_xf_\delta\circ p_x(v_1)+\delta v_2,\delta v_1)= T_xf_\delta\circ p_x(v) +\delta(v_2,v_1)\in \R^{N'} .\]
For every $x\in M$ and $\delta>0$, such a map is invertible, depends smoothly on $x$ and is $\delta$-close to $\R^{N'}\ni v\in \mathbb R^{N'} \rightarrow T_x f_\delta\circ p_x(v)\in \mathbb R^{N'}$. 
\end{proof}
\begin{corollary} The map $F^\delta$ is a homeomorphism of $\arr M_f\times \R^N$, for every $\delta>0$.\end{corollary}

For every $v\in \Gamma$, the following map is well defined:
\begin{displaymath}
  {F^\delta}_\star(v):= \underline x\in \arr M\mapsto {F^\delta}(v(\arr
  f^{-1}(\underline x))). 
\end{displaymath}
This map is continuous, linear and for $\delta>0$ it is bijective.

Moreover, we remark that ${F^\delta}_\star(v)$ belongs to $\Gamma$.
%
%
%
%
%

Now, let us suppose the existence of a right inverse $J$ of
${F^\delta}_\star
-id$.  This means:
\begin{displaymath}
  \big({F^\delta}_\star -id\big)J=id.
\end{displaymath}

We notice that (${C_1}$) is equivalent to find a fixed point $\phi\in
\Gamma$ of the operator:
\begin{displaymath}\Big[\big({F^\delta}_\star-id\big)-\big({\Phi_f^g}
  -id\big)\Big] \circ J=id -({\Phi_f^g} -id)\circ J,
\end{displaymath}
such that $w:= J(\phi)$ satisfies (${C_2}$) and (${C_3}$).


We construct $J$ in \textsection \ref{section de J}. From its
construction we get
\begin{proposition}
  \label{pro:contract-property}
  For every  $\epsilon>0$ and every $\eta>0$ sufficiently small w.r.t. $\epsilon$, there
  exists $\delta>0$ small enough such that for every $g$ $C^1$-close
  enough to $f$, the operator
  \begin{displaymath}
    \Big[\big({F^\delta}_\star-id\big)-\big({\Phi_f^g}
    -id\big)\big]\circ J= ({F^\delta}_\star -
    {\Phi_f^g}) \circ J
  \end{displaymath}
  is well defined on a $2\eta$-neighborhood of $0$ in $\Gamma$ and is
  $C^0$-contracting.
  
  Moreover, $\norm{\big({F^\delta}_\star-{ \Phi_f^g}\big) J
    v}_{C^0}\le \eta$ and
  $\Lambda\big(\big({F^\delta}_\star-{\Phi_f^g}\big)
  Jv\big)\leq\epsilon$, whenever $\norm{v}_{C^0}\le \eta$ and
  $\Lambda(v)\leq\epsilon$.
\end{proposition}

Together with Theorem \ref{part1}, this implies the inverse limit
structural stability of AS-endomorphisms.

\section{Construction of the right inverse \texorpdfstring{$J$}{J} of \texorpdfstring{${F^\delta}_\star
  -id$}{F-id}}
\label{section de J}
We recall that $f$ denotes an AS-endomorphism of a compact manifold
$M$. Let $\Omega(\arr f)$ be the non-wandering set of $\arr f$. It is
shown in \cite{BR12} that:
$$\Omega(\arr  f)=\arr M_f \cap \Omega(f)^\mathbb Z .$$ 

Moreover, the non-wandering set  $\Omega(\arr f)$ is the disjoint union of compact, transitive subsets $(\arr \Omega_i )_i$, called \emph{basic pieces}. The family of
all basic pieces is finite and called the \emph{spectral decomposition} of
$\Omega(\arr f)$.

For every basic piece $\arr \Omega_i$, we define the \emph{stable} and
\emph{unstable sets} of $\arr \Omega_i$, respectively, by
\begin{displaymath}
  W^s(\arr \Omega_i)=\{\underline x\in \arr M: \; d(\arr f^{n}(\underline
  x),\arr \Omega_i)\rightarrow 0, \text{ as }n\rightarrow +\infty\} 
\end{displaymath}
\begin{displaymath}
  W^u(\arr \Omega_i)=\{\underline x\in \arr M: \; d({\arr f}^{n}(\underline
  x),\arr \Omega_i)\rightarrow 0, \text{ as }n\rightarrow -\infty\} 
\end{displaymath}
The geometry of these sets is studied in \cite{BR12}.

Given two basic pieces $\arr \Omega_i$ and $\arr \Omega_j$, we write
$\arr \Omega_i\succ \arr \Omega_j$ if $W^u(\arr \Omega_i)$ intersects
$W^s(\arr \Omega_j)\setminus \arr \Omega_j$. In \cite{BR12} it is shown that
for any AS-endomorphism $f$, the relation $\succ$ is an order
relation. This enables us to enumerate the 
spectral decomposition $(\arr \Omega_i)_{i=1}^q$ of $\Omega(\arr f)$
in such a way that $\arr \Omega_i\succ \arr \Omega_j$ implies $i>j$.

We recall that a \emph{filtration adapted to $(\arr \Omega_i)_{i}$} is
an increasing sequence of compact sets
\[\varnothing=M_0\subset M_1\subset\cdots \subset M_i\subset \cdots
\subset M_q=\arr M_f\] such that for $q\ge i\ge 1$:
\[\bigcap_{n\in \Z} \arr f^n(M_{i}\setminus M_{i-1})= \arr
\Omega_i\quad \text{ and }\quad \arr f(M_i)\subset int(M_i).\] The
existence of such a filtration is shown in Corollary 4.7 of
\cite{BR12}. \label{filtration}

The following proposition is formally similar to the one used by
Robbin \cite{Ri} or Robinson \cite{Rs}, but it is technically much more
complicated and its proof requires to be handled very carefully. New ideas will be needed. The proof will be done
in \textsection \ref{propprincipal}-\ref{princis}-\ref{princiu} and will use \textsection\ref{analysis}.
\begin{proposition}
  \label{principal}
  There exist $\lambda\in(0,1)$, $K>0$ and an open cover
  $(W_i)_{i=1}^{q}$ of $\arr M_f$, where each $W_i$ is a neighborhood
  of $\arr \Omega_i$, and such that for every $\delta>0$, there
  exist vector subbundles $E_i^u$ and $E_i^s$ of the trivial bundle $W_i\times\R^N\to W_i$
  satisfying the following properties:
  \begin{itemize}
  \item[$(i)$] For every $\underline x\in W_i\cap \arr f^{-1}(W_i)$,
    the map $F^\delta$ sends $E_{i\underline x}^s$ and $E_{i\underline
      x}^u$ onto $E_{i\arr f(\underline x)}^s$ and $E_{i\arr
      f(\underline x)}^u$ respectively.
  \item[$(ii)$] For any $k\ge j$ and every ${\underline x}\in W_k\cap
    \arr f^{-1}(W_j)$, the following inclusions hold:
    \begin{displaymath}
      F^\delta (E^s_{k\underline x})\subset E^s_{j\arr f(\underline x)},\quad F^\delta(E^u_{k\underline x})\supset E^u_{j\arr f(\underline x)}.
 \end{displaymath}
  \item[$(iii)$] $E^s_{i\underline x}\oplus E^u_{i\underline x}=\R^N$,
    for any $i\in\{1,\ldots,q\}$ and every $\underbar x\in W_i$; the angle between $E_i^s$ and $E_i^u$ is bounded from
    below by $K^{-1}$.
  \item[$(iv)$]   The subbundles $E_i^u$ and $E_i^s$ are $d_1$-continuous and
    $d_\infty$-Lipschitz.
  \item[$(v)$] For every $i$ and any $\underline x\in W_i$, it holds
    \begin{displaymath}
      \|{F^\delta}(v^u)\|\ge \|v^u\|/K, \quad\forall v^u\in
      E^u_{i\underline x}. 
    \end{displaymath}
  \item[$(vi)$] For every $q'$, if $\underline x\in W_{q'}$ then $\arr f(\underline x)\notin \cup_{j>q'} W_j$ and $\cap_{n\in \mathbb Z} \arr f^n(W_{q'})=\arr \Omega_{q'}$.
  \item[$(vii)$] For every $i$, any $\underline x$ in a
    neighborhood of $\arr \Omega_i$ which does not depend on $\delta$,
    and for all $v^s\in E^s_{i\underline x}$ and $v^u\in
    E^u_{i\underline x}$, it holds:
    \begin{displaymath}
      \|{F^\delta}(v^s)\|\le \lambda \|v^s\|\quad \mathrm{and}\quad
      \|{F^\delta}(v^u)\|\ge \|v^u\|/\lambda.
    \end{displaymath}
  \end{itemize}
\end{proposition}
The subbundles $E_i^u$ and $E_i^s$ can be  considered as functions from
$W_i$ to the Grassmannian $G_N$ of $\R^{N}$. Property $(iv)$ means that they are $d_1$-continuous
and $d_\infty$-Lipschitz.  The Grassmanian $G_N$ is a manifold with as
many connected components as possible dimension for $\R^N$-subspaces, \ie $N+1$. \label{notation for
  mathcal G}
\begin{remark}A main difficulty in this proposition is that $K$ does not depend on $\delta$, whereas the norm of the inverse of $F^\delta$ blows up as $\delta$ approaches 0 whenever $f$ has critical points. Hence the proof of this proposition will not be symmetric in $u$ and $s$.
\end{remark}
We will prove in Corollary \ref{partition1} the existence of a
partition of the unity $(\gamma_i)_i$
 subordinated to $(W_i)_{i=1}^q$, where each $\gamma_i$ is $d_1$-continuous and $d_\infty$-Lipschitz.

Given any $\underline x\in W_i$, let $\pi_{i\underline
  x}^s:\R^N\rightarrow E^s_{i\underline x}$ denote the projection
parallely to $E^u_{i\underline x}$ and $\pi_{i\underline x}^u :\R^N
\rightarrow E^u_{i\underline x}$ the projection parallely to
$E^s_{i\underline x}$.

For $v\in \Gamma$ and $\sigma\in \{s,u\}$, put:
\begin{equation}
  \label{eq:Ji-def}
  v_i^\sigma:= \pi_i^\sigma(\gamma_i\cdot v),
  \quad J_{is}(v):=- \sum_{n=0}^\infty {F^\delta}_{ \star }^n(v_i^s)\quad
  \text{and}\quad J_{iu}(v):= \sum_{n=-\infty}^{-1} {F^\delta}_{ \star
  }^n(v_i^u). 
\end{equation}
We can now define:
\begin{equation}
  \label{eq:J-def}
  J:= \sum_{i, \sigma} J_{i\sigma}.
\end{equation}

Let us define
\begin{displaymath}
  C:=\sup_{1\leq i\leq q}\sup_{\underline{x}\in\mathrm{supp}\gamma_i}
  \left\{\|\pi_{i\underline x}^s\|,\|\pi_{i\underline
      x}^u\|\right\}. 
\end{displaymath}
By Property $(iii)$ the constant $C$ is bounded from above independently of $\delta$.

\begin{lemma}\label{normJC0}
There exists a constant $D$ independent of $\delta$ such that for every $j$, for all $\underline x\in W_j$ and $u\in E_{j\underline x}^s$ (resp. $u\in E_{j\underline x}^u$), for every $n\ge 0$ (resp. $n\le 0$): 
\[\|(F^\delta)^n (u)\|\le D \lambda^{|n|} \|u\|\]
\end{lemma}
\begin{proof}
For every $\underline x\in W_j$, since $(W_k)_k$ is a cover of $\arr M_f$, there exists a sequence $(n_i)_i$ such that $\arr f^i(\underline x)\in W_{n_i}$ for every $i$. From property $(vi)$, the sequence $(n_i)_i$ must be decreasing. 

As $\underline x\in  W_j$,  we can suppose that $n_0= j$.
By property $(ii)$, for every $k\ge 0$, ${F^\delta}^{-k}$ sends $E_{j\underline x}^u$ into  $E_{n_k}^u$ and ${F^{\delta}}^k $ sends $E_{j\underline x} ^s$ into $E^s_{n_k}$.

Let $(V_k)_k$ be the neighborhoods of respectively $(\arr \Omega_k)_k$ on which $(vii)$ holds. Since the non-wandering set contains the limit set and $\arr M_f$ is compact,  there exists $m\ge 0$ such that there is no $\underline x\in \arr M_f$ such that $(\underline x, \arr f(\underline x), \cdots \arr f^{m-1} (\underline x))$ are all outside of $\cup_k V_k$. We can suppose $V_k$ included in $W_k$ for every $k$. 
Consequently, for every $\underline x$, all the terms $f^i(\underline x)$ of the sequence $(f^i(\underline x))_i$ but $qm$ are in $V_{n_i}$. From $(vii)$.
For all $\underline x\in W_i$ and $u\in E_{i\underline x}^s$ (resp. $u\in E_{i\underline x}^u$), for every $n\ge 0$ (resp. $n\le 0$): 
\[\|(F^\delta)^n (u)\|\le D \lambda^{|n|} \|u\|,\]
with $D= (\max (\|F^\delta\|, \|(F^\delta|E^u_{i\underline x})^{-1}\|  )^{qm}$, which is bounded by a constant independent of $\delta$ by $(v)$. 
\end{proof}
From Lemma \ref{normJC0}, it holds that for every $v \in \Gamma$:
\begin{equation}\label{normeinde} \|Jv\|_{C^0}\le \frac{2 C Dq}{1-\lambda} \|v\|_{C^0},\end{equation}
where $C$, $D$ and $\lambda$ are independent of $\delta>0$ small.

Moreover we easily compute the following.
\begin{proposition}
  The map $J$ is the right inverse of ${F^\delta}_{ \star }-id$:
  \begin{displaymath}
    ({F^\delta}_{ \star }-id)\circ J=id.
  \end{displaymath}
\end{proposition}

To prove main Theorem \ref{thm:main}, it remains only to prove Propositions~\ref{pro:contract-property} and \ref{principal}. 

To show Proposition \ref{principal}, we will develop some analytical tools in the next section. On the other hand, we are ready to prove
Proposition~\ref{pro:contract-property}.

\begin{proof}[Proof of Proposition~\ref{pro:contract-property}:]
  Let us start by computing $({F^\delta}_\star-{\Phi_f^g})$ at $0\in
  \Gamma$:
  \begin{equation}
    \label{eq:dist-to-origin}
    \begin{split}
      ({F^\delta}_\star-{\Phi_f^g})(0) (\underline
      x)&=F^\delta_{x_{-1}}(0)-\exp_{x_0}^{-1}\circ
      g\circ\exp_{x_{-1}}(p_{x_{-1}}(0)) \\
      &=-\exp_{x_0}^{-1}(g(x_{-1})).
    \end{split}
  \end{equation}

  In particular, this implies $\norm{\big({F^\delta}_\star
    -{\Phi_f^g}\big)(0)}_{C^0}=d_{C^0}(f,g)$.

  On the other hand, $({F^\delta}_\star-{\Phi_f^f})$ is a $C^1$ map defined on a neighborhood of $0\in\Gamma$. Thus we can compute its derivative at the origin:

\begin{equation}
    \label{eq:deriv-F-phiff-0}
    D_0\big({F^\delta}_\star-{\Phi_f^f}\big)
    (v)(\underline x) = \left(F^\delta_{x _{-1}}-
      T_{x_{-1}}f\right) (p_{x_{-1}}\circ v\circ \arr
    f^{-1}(\underline x)) 
  \end{equation}
for every $v\in\Gamma$ and every $\underline x\in\arr M_f$. In particular, the operator norm subordinate to the $C^0$ norm satisfies: 
  \begin{displaymath}
    \norm{D_0\big({F^\delta}_\star-{\Phi_f^f}
      \big)}_{C^0}\rightarrow 0, \; \text{as}\; \delta \rightarrow 0.
  \end{displaymath}
At a neighborhood of $0$, the derivative of $F_\star ^\delta$ is constant, whereas the one of $\phi_f^f$ is continuous. Furthermore, for $g$ $C^1$-close to $f$, $D\phi_f^g$ is close to $D\phi_f^f$.

  Hence, for every $\mu>0$, there exists a small $\eta (\mu)>0$ such that for any $g$ sufficiently close to $f$ in the $C^1$-topology and any $w\in\Gamma$ with
  $\norm{w}_{C^0}\leq\eta(\mu)$ and $\delta\le \eta(\mu)$, it holds
  \begin{equation}
    \label{eq:deriv-estimate}
    \norm{D_w\big({F^\delta}_\star-{\Phi_f^g}
      \big)}_{C^0}\leq \mu.
  \end{equation}

  Then, putting together 
  \eqref{eq:dist-to-origin}, \eqref{eq:deriv-estimate},  we get
  \begin{displaymath}
    \begin{split}
      \norm{({F^\delta}_\star-{\Phi_f^g})Jv}_{C^0} &\leq d_{C^0}(f,g)
      +\mu \norm{Jv}_{C^0} \\
      &\leq d_{C^0}(f,g) +
      \norm{J}_{C^0}\mu
      \norm{v}_{C^0}.\\
      \end{split}
  \end{displaymath}
      
By (\ref{normeinde}), $\norm{J}_{C^0}$ is bounded independently of $\delta$, we put
\begin{equation}\label{conditionsurmu0}[\mu_0= \inf_{\delta \; \text{small}} \min\left(\frac12, \frac1{2\norm{J}_{C^0}}\right)>0.\end{equation}
 Hence for every $\eta,\delta<\eta(\mu_0)$, for every $g$ moreover $\eta/2$-$C^0$-close to $f$ it holds for $v\in \Gamma$:
 \[\|v\|_{C^0} \le \eta \Rightarrow       \norm{({F^\delta}_\star-{\Phi_f^g})Jv}_{C^0} \le \frac \eta2 + \frac12 \|v\|_{C^0}< \eta.\]
Which is the second statement of the Proposition. Also inequalities (\ref{eq:deriv-estimate}) and (\ref{conditionsurmu0}) implies that  $({F^\delta}_\star-{\Phi_f^g})J$ contracts the $C^0$-norm by a small factor when $\eta$, $\delta$ are small and $g$ is close to $f$, which is the first statement of the Proposition. 
%
We remark that as far as $\delta\le \eta(\mu_0)$, which does not depend on $\eta$, we can suppose $\eta\le \eta(\mu_0)$ as small as we want which satisfies the same property, if $g$ is sufficiently close to $f$.

It remains only to estimate $\Lambda(({F^\delta}_\star-\Phi_f^g)(Jv))$ for $v\in \Gamma$. To
  do that, we prove the following lemma similar to the Robin's computation \S6 of \cite{Ri}:
  \begin{lemma}\label{LemmaRobin}
For $\sigma\in \{s,u\}$, there exist a constant $A$ which depends on $f$ but not on $\delta$, and a constant $B_\delta$ which depends on $\delta$ such that for every $v\in \Gamma$,   for any $i$ and $\sigma=s,u$: 
  \begin{equation}
    \label{eq:Lambda-Ji}
    \Lambda(J_{i\sigma}v_i^\sigma)\leq
  A\Lambda(v_i^\sigma) +
   B_\delta \norm{v^\sigma_i}_{C^0}. 
  \end{equation}
  \end{lemma}
As the norm of $\Lambda(v_i^\sigma)$ is dominated by $\Lambda(v)$ times a constant independent of $\delta$, it holds by taking the constants $A$ and $B_\delta$ larger:
  \begin{equation}
    \label{eq:Lambda-J}
    \Lambda(Jv)\leq
    A\Lambda(v) +
   B_\delta \norm{v}_{C^0}.   
  \end{equation}
Put $L_{\underline x} := F^\delta_{x_{-1}}- \exp^{-1}_{x_0}\circ g\circ \exp_{x_{-1}}\circ p_{x_{-1}}$.  
We have:
\[ ({F^\delta}_\star-\Phi_f^g)J(v)(\underline x) -({F^\delta}_\star-\Phi_f^g)J(v)(\underline y)
= L_{\underline x} (J(v)(\underline x) -J(v)(\underline y))+
 (L_{\underline x} -L_{\underline y})J(v)(\underline y)\] 
Hence:
 \[\Lambda(({F^\delta}_\star-\Phi_f^g)J(v)) \le  \|L\|_{C^0} (A\Lambda(v) + B_\delta \norm{v}_{C^0})+ \Lambda(L)\|J(v)\|_{C^0},
   \]
 where  $\Lambda(L)$ depends on $\delta$.   
 
By (\ref{eq:deriv-estimate}), we can suppose $g$ sufficiently close to $f$ and $\delta$ small enough so that $\|L\|_{C^0} $ is  $1/(2 A)$ contracting on a small neighborhood of $0$.  From this: 
 \[\Lambda(  ({F^\delta}_\star-\Phi_f^g)J(v)) \le  \frac{\Lambda(v)}2 +
  \big(\frac{ B_\delta}{2A}  + \Lambda(L) \|J\|_{C^0}\big) \norm{v}_{C^0},\]
Hence for every $\epsilon>0$, for every $\eta$ such that:
\[\eta\le (\frac{ B_\delta}{2A}  + \Lambda(L) \|J\|_{C^0})^{-1}\frac\epsilon2\]
If $\norm{v}_{C^0}\le \eta$ and $\Lambda(v)\le \epsilon$ and $g$ sufficiently close to $f$, it holds:
  \[\Lambda(  ({F^\delta}_\star-\Phi_f^g)J(v)) \le \epsilon\]
 
\end{proof}

\begin{proof}[Proof of Lemma \ref{LemmaRobin}]
We prove the case $\sigma=s$, since the other case $\sigma=u$ is similar. 
For $n\ge 0$,  we evaluate:
  \[\|{F^\delta}^n (\underline x, v_i^s(\underline x))-{F^\delta}^n ({\underline y}, v_i^s(\underline y))\|\]
  \[\le \|{F^\delta}^n \circ \pi_i^s(\underline x, v_i^s(\underline x)- v_i^s(\underline y))\|+
  \|{F^\delta}^n \circ \pi_i^s (\underline x,v_i^s(\underline y) )-{F^\delta}^n \circ \pi_i^s(\underline y,v_i^s(\underline y) )\|\]
  By remark \ref{normJC0}, there exists a constant $D$ which does not depend on $n$ nor $\delta$ such that:
  \[\|{F^\delta}^n |E^s_i\|\le D \lambda^{n}\]
  Hence:
  \[\|{F^\delta}^n \circ \pi_i^s(\underline x, v_i^s(\underline x)- v_i^s(\underline y))\|\le DC\lambda^{n} C\Lambda (v_i^s) d_\infty(\underline x, \underline y)\]

On the other hand,
     \[\|{F^\delta}^n \circ \pi_i^s (\underline x,\cdot )-{F^\delta}^n \circ \pi_i^s(\underline y,\cdot )\|
  \le \]
  \[\sum_k \| {F^\delta} ^{n-k-1} |E^s_{n_k \arr f ^{k+1}(\underline x)}\|\cdot \| F^\delta (\arr f ^k(\underline x),\cdot ) -
  F^\delta (\arr f ^k(\underline y),\cdot )\|\cdot \|Tf ^k |E^s_{\underline y}\|, \]
where $n_k$ is such that $\arr f^{k+1}(\underline x)\in W_{n_k}$. This is less than:
\[  \sum_k CD\lambda^{n-k-1} \cdot \| F^\delta (\arr f ^k(\underline x),\cdot ) -
  F^\delta (\arr f ^k(\underline y),\cdot )\|\cdot CD \lambda^k.\]
 Hence there exists a constant $K(\delta)$ which depends only on $f$ and $\delta$ such that: 
     \[\|{F^\delta}^n \circ \pi_i^s (\underline x,\cdot )-{F^\delta}^n \circ \pi_i^s(\underline y,\cdot )\|
  \le  n\cdot K(\delta) \lambda^{n} d_\infty (\underline x, \underline y)\]
  Consequently :
   \[\|{F^\delta}^n (\underline x, v_i^s(\underline x))-{F^\delta}^n ({\underline y}, v_i^s(\underline y))\|\le 
 (DC^2\lambda^{n} \Lambda (v_i^s)+ n\cdot K(\delta) \lambda^{n}\|v\|_{C^0}) d_\infty (\underline x, \underline y)\]
  Summing over $n$ we conclude.
  \end{proof}

\section{Analysis on $M_f$}\label{analysis}
Let us introduce a few notations. Let $N$ be an arbitrary
Riemannian manifold. We recall that $C^0(\arr M_f, N)$ denotes the
space of $d_1$-continuous maps $\phi:\arr M_f\rightarrow N$.  Also
$\Lip^\infty (\arr M_f, N)$ denotes the space of $d_\infty$-Lipschitz
maps $\phi:\arr M_f\rightarrow N$. Let us define:


\begin{displaymath}
  \Mor^\infty_0(\arr M_f,N):=C^0 (\arr M_f, N)\cap\Lip^\infty (\arr
  M_f, N). 
\end{displaymath}
We endow $C^0(\arr M_f, N)$ with the uniform distance given by the
Riemannian metric of $N$. Note that $C^0(\arr M_f, N)$ is a Banachic
manifold.  Actually its topology does not depend on the Riemannian
metric of $N$. The aim of this section is to prove the denseness of
$\Mor^\infty_0(\arr M_f,N)$ in $C^0 (\arr M_f, N)$.  To do this, we
will use a new technique based on convolutions. 

Let $\rho\in C^\infty (\R)$ be a non-negative bump function with
support in $(-1,1)$. Let $\mu$ be any Lebesgue measure on $M$ such
that $\mu(M)=1$, and let $\tilde \mu=\bigotimes_{\mathbb Z}\mu$ be the
induced probability on $M^\mathbb{Z}$.

For every map $g$ from $\arr M_f$ into $\R^n$, for every $r>0$, we define $g_r$ by: 
  \begin{displaymath}
    g_r:\; \arr M_f\ni\underline x \mapsto
    {\int\limits_{M^\mathbb Z} g(\underline y)\cdot \rho
      \left(\frac{d_1 (\underline x,\underline y)}r\right)d\tilde
      \mu(\underline y)}. 
  \end{displaymath}

The following result plays a key role:
\begin{lemma}
  \label{covol}
  Let $\phi:\; \arr M_f\rightarrow \R^n$ be a continuous function with
  respect to the distance $d_1$. Let $\tilde \phi$ be a continuous
  extension to $(M^\mathbb Z, d_1)$.  Let $1\!\!1$ be the function on
$ M^{\mathbb Z}$ constantly equal to $1\in \mathbb R$. For every $r>0$, the functions $1\!\!1_r$ and $\tilde \phi_r$ (defined as above) satisfy:

  \begin{itemize}
  \item[\it (i)] $\tilde \phi_r$ and $\tilde \phi_r/1\!\!1_r$ are well
    defined.
  \item[\it (ii)] $\tilde \phi_r$ is $d_1$-continuous and $d_\infty$-Lipschitz, \ie it belongs to $\Mor^\infty_0(\arr M_f, \mathbb
    R^n)$.
  \item[\it (iii)] The function $\tilde \phi_r/1\!\!1_r$ is
    $C^0$-close to $\tilde \phi$ whenever $r$ is small.
  \item[\it (iv)] The support of $\tilde \phi_r$ is included in the
    $r$-neighborhood of the support of $\phi$.
  \end{itemize}
\end{lemma}
The following are immediate corollaries of this lemma:
\begin{corollary}
  \label{partition1}
  For every open cover $(U_i)_i$ of $\arr M_f$, there exists a
  partition of unity $(\rho_i)_i\subset\Mor^\infty_0( \arr M_f,\mathbb
  R)$ subordinate to it.
\end{corollary}

\begin{corollary}
  \label{approximation general}
  The subset $\Mor^\infty_0( \arr M_f,N)$ is dense in $C^0(\arr M_f,
  N)$.
\end{corollary}
\begin{remark}Both above corollaries are also true if we replace  $\arr M_f$ by  any compact subset $E$ of it.
\end{remark}

\begin{proof}[Proof of Lemma \ref{covol}]
  Let us start by proving {\it (i)}. As $\tilde \phi$ and $\rho$ are
  continuous on a compact space, they are bounded. As $\tilde
  \mu(M^\mathbb Z)=1$, the functions $\tilde \phi_r$ and $1\!\!1_r$
  are well defined. Let $\underline x\in M^\mathbb Z$. There exists
  $\delta>0$ such that $\rho|B_{d_1}(0,\delta/r)$ is greater than
  $\delta$. For every $\underline x\in M^\mathbb Z$, the $\tilde
  \mu$-volume of the ball $B_{d_1}(\underline x,\delta/r)$ is greater
  than:
  \begin{displaymath}
    \prod_{-N}^N  \mu \bigg(B\Big(x_i, \frac{\delta}{6r}\Big)\bigg)>0,
  \end{displaymath}
  where $N$ is any natural number satisfying $\sum_{|n|\ge N}
  2^{-|n|}diam(M)\le \frac{\delta}{2r}$.

  Thus, $m:= \inf\{\tilde \mu(B_{d_1}(\underline x, \delta/r)):\;
  \underline x\in \arr M\}$ is positive and $1\!\!1_r> m \delta$.
  Consequently, $\phi_r/1\!\!1_r$ is everywhere well defined.

  Let us proof {\it (iii)}. As $(M^{\mathbb Z},d_1)$ is compact, the
  function $\tilde \phi$ is uniformly continuous: for every $\delta>0$,
  there exists $r>0$ such that the image by $\tilde \phi$ of any
  $d_1$-ball of radius $r$ has diameter less than $\delta$.  Thus for every $\underline x\in \arr M_f$:
  \begin{displaymath}
    |\tilde \phi_r(\underline x)-\tilde \phi(\underline x)\cdot
    1\!\!1_r(\underline x)|\le {\int\limits_{M^{\mathbb Z}} \delta
      \cdot \rho \left(\frac{d_1 (\underline x,\underline
          y)}r\right)\tilde \mu(\underline y)}\le\delta \cdot
    1\!\!1_r(\underline x) 
  \end{displaymath}

  Let us proof {\it (ii)}. We remark that if a function is
  $d_1$-Lipschitz, then it is $d_\infty$-Lipschitz, and so it
  belongs to $\Mor^\infty_0(\arr M_f, \mathbb R^m)$.  Then, let us
  prove $\tilde \phi_r$ is $d_1$-Lipschitz. For every $\underline
  x'\in \arr M_f$:
  \begin{displaymath}
    \tilde \phi_r(\underline x)- \tilde \phi_r(\underline x')=
    {\int\limits_{M^\mathbb Z} \tilde \phi(\underline y)\cdot
      \left(\rho \left(\frac{d_1 (\underline x,\underline
            y)}r\right)-\rho \left(\frac{d_1 (\underline x',\underline
            y)}r\right)\right)d\tilde \mu(\underline y)} 
  \end{displaymath}
  As $\rho$ is smooth, its derivative is bounded by some $L$, and so:
  \begin{displaymath}
    \left|\rho \left(\frac{d_1 (\underline x,\underline
          y)}r\right)-\rho \left(\frac{d_1 (\underline x',\underline
          y)}r\right)\right|\le \frac Lr |d_1 (\underline x,\underline
    y)-d_1 (\underline x',\underline y)|\le \frac Lr d_1 (\underline
    x,\underline x') 
  \end{displaymath}
  Consequently:
  \begin{displaymath}
    |\tilde \phi_r(\underline x)- \tilde \phi_r(\underline x')|\le
    \frac Lr d_1 (\underline x,\underline x')\int\limits_{M^\mathbb Z}
    |\tilde \phi(\underline y)| d\tilde \mu(\underline y) 
  \end{displaymath}
  Thus, since $\tilde \phi$ is bounded and $\tilde \mu$ is a
  probability, we get that $\tilde \phi_r$ is $d_1$-Lipschitz as desired.
\end{proof}

\section{Proof of Proposition \ref{principal}}\label{propprincipal}
Let $f$ be an $AS$ endomorphism of a compact manifold $M$.
\subsection{Preliminaries}\label{preli}
\paragraph{Distance on Grassmannian bundles.}
We endow the space of linear endomorphisms of $\R^N$ with the operator
norm $\|\cdot\|$ induced by the Euclidean one of $\R^N$. We recall
that the Grassmannian $G_N$ of $\R^N$ is the space of $d$-planes of
$\R^N$, for $0\le d\le N$. Given two planes $P,P'\in G_N$ let $\pi_P$ and
$\pi_P'$ be their associated orthogonal projections. The metric $d_G$ on $G_N$ is
defined by:
\begin{displaymath}
  d_G(P,P')=\norm{\pi_P-\pi_{P'}}
\end{displaymath}


\paragraph{Angle between planes.}
Two planes $P$ and $P'$ of $\R^n$ make an angle greater than $\eta$ if for all $u\in P\setminus \{0\}$ and $v\in P'\setminus \{0\}$, the angle between $u$ and $v$ is greater than $\eta$ (for the Euclidean norm), in particular they are in direct sum.
\paragraph{Definition of $E^s$.} 
 We recall that for any $(\underline x, a)\in\arr M_f\times \R^N$,
${F^\delta}(\underline x, a)=(\arr f(\underline x),
F^\delta_{x_0}(a))$, where $x_0=\pi_0(\underline x)$.

The stable direction $E^s_{\underline x}$ of $F^0 $ at $\underline x$ is given by
\begin{equation}
  \label{eq:Es-def-for-F0}
  E^s_{\underline x}:= \text{Ker }p_{x_0}\oplus T_{x_0} W^s(x_0, f),
\end{equation}
where $W^s(x_0,f)$ is the stable set of $x_0$; its intersection with a neighborhood of $\pi_0(\arr M_f)$ is an immersed manifold (see Prop. 4.11 \cite{BR12}).

We remark that $E^s_{\underline x}$ depends only on $x_0$, for every $\underline x\in \arr M_f$. 

In order to construct the plane fields of
Proposition~\ref{principal}, we will have to take care of the critical
points of $f$. The unique control that we have on them is the strong
transversality condition. This condition implies, in particular, that
for every $\underline x\in \arr M_f$ it holds
\begin{equation}\label{tune}
  E^s_{\arr f(\underline x)}+ Tf(T_{ x_0} M)= \R^N.
\end{equation}

Therefore we shall construct the distributions $(E_i^s)_{i=1}^q$
``close'' to $E^s$ in $G_N$. Let us explain how we will proceed, and what does
it mean.

\paragraph{Topology on plane fields of nested domains of definition}
For a subset $C\subset \arr M_f$, we denote by $C^0(C, G_N)$ the space of $d_1$-continuous maps from $C$ into $G_N$. When $C$ is compact,  we endow this space with the uniform metric:
\begin{displaymath}
  d(g,g')= \max_{\underline x\in C} d(g(x), g'(x)).
\end{displaymath}
Given a plane field $E\in C^0(C, G_N)$ and $\eta>0$, we denote by $B(E,\eta)$ (resp. $\bar B(E, \eta)$) the open (resp. closed) ball centered at $E$ and radius $\eta$.

Let $W$ be subset of $\arr M_f$ and $V$ a neighborhood of $W$. Let $E_W\in C^0(W,G_N)$ and $E_V\in C^0(V,G_N)$ be two plane fields.
We say that $E_W$ is \emph{compact-open close} to $E_V$ if for any compact subset $C\subset W$, there exists a small compact neighborhood $N$ of $C$ in $V$ such that the graph of $E_V| N$ is close to the graph of $E_W|C$ for the Hausdorff distance on compact subsets of $M_f\times G_N$ induced by $d_1+d_G$.
This will be explained in greater details for its application case in remark \ref{defdistancesurensemblesdifferents}.

\subsection{Splitting Proposition \ref{principal} into the stable and unstable fields}  
We are going to illustrate the geometrical part of the proof of Proposition \ref{principal} by depicting the construction for the following example.  Let $f\colon (x,y,z)\in \R^3\mapsto (x^2,y^2,0)$. This map is AS and can be extended to an AS endomorphism of the compactification $(\R\cup \{\infty\})^3$ of $\R^3$ equal to the 3-torus. On this compact manifold, its inverse limit is homeomorphic to $[0,\infty]^3$, via the projection $\pi_0$. Since this map is invariant via the symmetries $(x,y,z) \mapsto (x^{\delta_x},y^{\delta_y},z^{\delta_z})$, $(\delta_x,\delta_y,\delta_z)\in \{-1,1\}^3$, we will focus only on the restricted dynamics on $\pi_0^{-1}([0,1]^3)$ which is the inverse limit of $f$ restricted to the set of points with bounded orbit. The restricted non-wandering set $\arr \Omega_f$ is formed by $4$ fixed points $
(0,0,0)^\Z$, $(0,1,0)^\Z$, $(1,0,0)^\Z$, $(1,1,0)^\Z$.

Let us split Proposition \ref{principal} into two propositions. 

\begin{proposition}\label{fondaprops}
There exist neighborhoods  $(V_i)_{i=1}^q$  of respectively $(W^s(\arr \Omega_i))_{i=1}^q$ in $\arr M_f$, and for every small $\delta$, there are functions 
\[E_{i}^s\colon V_i\to G_N\]
satisfying the following properties for every $i$:
\begin{enumerate}[$(i)$]
\item for every ${\underline x}\in V_i\cap \arr f^{-1}(V_i)$ the following inclusion holds:
\[F^\delta(E_{i{\underline x}}^s)\subset E_{i\arr f({\underline x})}^s.\]
\item for every $k\ge j$, for every ${\underline x}\in V_k\cap \arr f^{-1}( V_j)$ the following inclusion holds:
\[F^\delta (E_{k {\underline x}}^s)\subset E_{j\arr f({\underline x})}^s.\]
\item $E_{i}^s$ is compact-open close to $E^s|W^s(\arr \Omega_i)$, when $\delta$ is small. 
%
\item   $E_{i}^s$ is of constant dimension, $d_1$-continuous, and locally Lipschitz for the  metric $d_{\infty}$.
\end{enumerate}
\end{proposition}
Figure \ref{EideltaS} depicts an example of such plane fields.

\begin{figure}[h]
    \centering
        \includegraphics{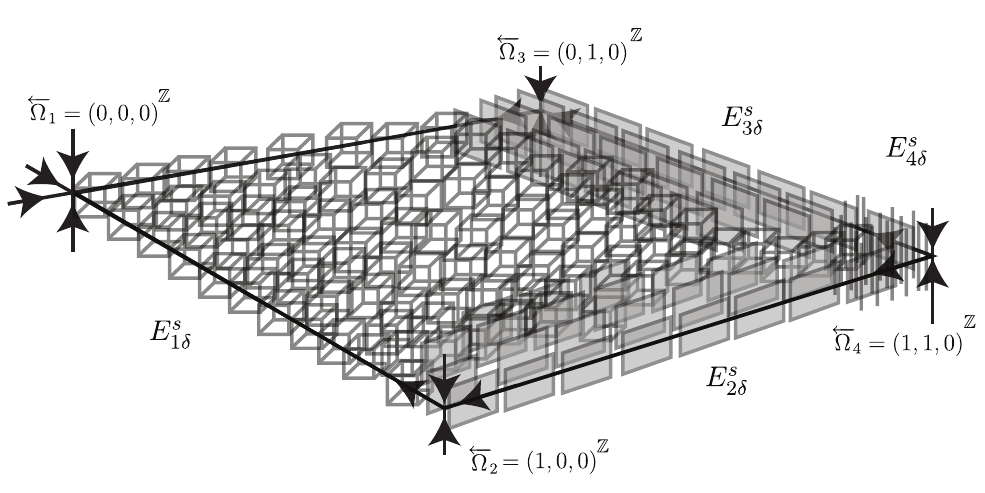}
    \caption{Plane fields $E_{i}^s$ in the example given by $f\colon (x,y,z)\mapsto (x^2,y^2, 0)$. }
\label{EideltaS}
\end{figure}

\begin{remark}\label{defdistancesurensemblesdifferents}
From the definition given in \textsection \ref{preli}, Property $(iii)$ means that for every $i$, for every compact subset $C$ of $W^s(\arr \Omega_i)$, for every $\epsilon>0$, there exists a compact neighborhood $U$ of $C$ in $V_i$ such that for every $\delta$ sufficiently small 
\[d_H(Graph(E_{i}^s|U), Graph(E^s|C))\le \epsilon,\]
where $d_H(\cdot,\cdot)$ denotes the Hausdorff distance of compact subsets of $\arr
M_f\times G_N$ induced by the distance $d_1+ d_G$. We notice that $U$ depends on $C$ and $\epsilon$ but not on $\delta$ small enough.
\end{remark}
\begin{remark}\label{defproche} Property $(iv)$ means that $E_{i}^s$ is of constant dimension, $d_1$-continuous, and that for every compact subset $C$ of $V_i$ there exists a constant $L_C^\delta$ such that:
\[d_G(E_{i\underline x}^s, E_{i\underline y}^s)\le L_C^\delta d_\infty(\underline x, \underline y),\quad \forall \underline x, \underline y\in C.\]
\end{remark}
 
In the diffeomorphism case, to obtain the existence of $(E^u_{i\delta})_i$ it suffices to first push forward by $F^\delta$ each of the plane field $E^s_{i}$ on $\cup_n \arr f^n(V_i)$ (which is a neighborhood of $W^u(\arr \Omega_i)$, and then to apply the same proposition to $\arr f^{-1}$. In our case, even though $\arr f$ is invertible, the bundle map $F^0$ is not. However, in Lemma \ref{Finversible}, we saw that $F^\delta$ is invertible for every $\delta>0$. Nonetheless, the norm of the inverse of this map depends on $\delta$, and so the angle between $E_{i}^s$ and $E_{i}^u$ as well. However in Proposition \ref{principal} such an angle must be bounded by a constant which is independent of $\delta$ (and this is necessary in the proof of Proposition \ref{pro:contract-property}). 

Hence we must redo a similar construction, still in a neighborhood of each $W^s(\arr \Omega_i)$ since it is the only place where we control  the singularities. 

Another difference in the construction of $E_{i}^u$ is the following: to construct the plane field $E_{i}^u$ we will not be allowed to pull back, since the critical set might intersect $W^s (\arr \Omega_i)$, and a pull back by $F^0$ would contain critical vectors which belong to $E^s$, this would contradict the angle condition $(iii)$ for $\delta>0$. Hence the construction of $E_{i}^u$  must be done in compact set in a small neighborhood of $W^s(\arr \Omega_i)$ via push forward.

\begin{proposition}
  \label{fondapropu} There exist $K>0$, an open cover
  $(W_i)_{i=1}^{q}$ of $\arr M_f$, where each $W_i$ contains $\arr \Omega_i$ and is included in $V_i$, 
  such that for every $\delta>0$,
 there exists a subbundle $E_i^u$ of $W_i\times\R^N\to W_i$
  satisfying the following properties for $i\in [1,q]$:
  \begin{enumerate}[(i)]
  \item For every $\underline x\in W_i\cap \arr f^{-1}(W_i)$,
    the map $F^\delta$ sends $E_{i\underline
      x}^u$ into $E_{i\arr
      f(\underline x)}^u$.
  \item For every $j\ge i$, for every $\underline x\in W_i\cap \arr f^{-1}(W_j)$ the following inclusion holds:
\[F^\delta (E_{j{\underline x}}^u)\supset E_{i\arr f({\underline x})}^u.\]
  \item $E^s_{i\underline x}\oplus E^u_{i\underline x}=\R^N$,
    for every $\underbar x\in W_i$,
the angle between $E_{i}^s$ and $E_{i}^u$ is bounded from
    below by $K^{-1}$.
  \item The subbundle $E_{i}^u$ is of constant dimension, $d_1$-continuous and
    $d_\infty$-Lipschitz.
  \item For every $\underline x\in W_i$, it holds
    \begin{displaymath}
      \|{F^\delta}(v^u)\|\ge \|v^u\|/K, \quad\forall v^u\in
      E^u_{i\underline x}. 
    \end{displaymath}
  \item[$(vi)$] For every $q'$, if $\underline x\in W_{q'}$ then $\arr f(\underline x)\notin \cup_{j>q'} W_j$ and $\cap_{n\in \mathbb Z} \arr f^n(W_{q'})=\arr \Omega_{q'}$.
  \end{enumerate}
\end{proposition}
Figure \ref{EideltaU} depicts an example of such plane fields. 

\begin{figure}[h]
    \centering
        \includegraphics{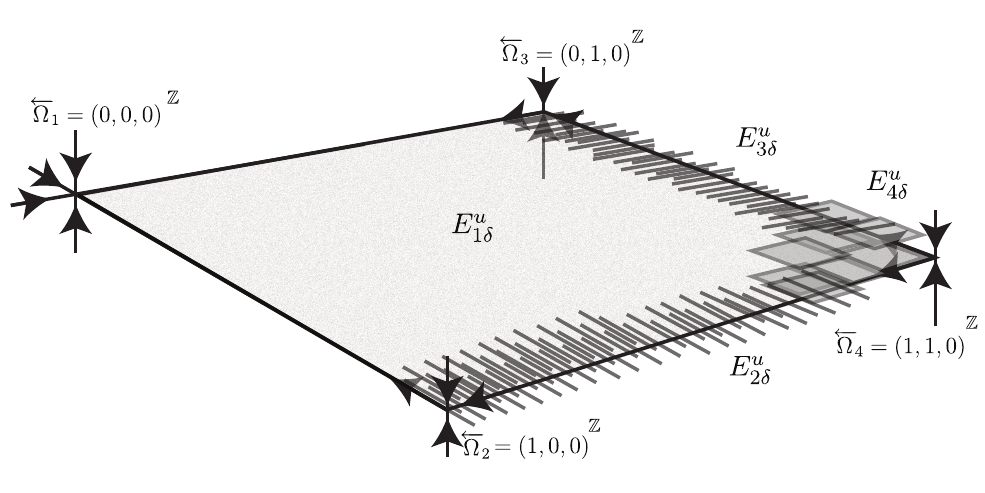}
    \caption{Plane fields $E_{i}^u$ in the example given by $f\colon (x,y,z)\mapsto (x^2,y^2, 0)$.}
\label{EideltaU}
\end{figure}

From the two latter propositions,  we easily deduce:  
\begin{proof}[Proof of  Proposition \ref{principal}]
By Propositions \ref{fondaprops} and \ref{fondapropu}, we have immediately properties $(i)$-$(ii)$-$(iii)$-$(iv)$-$(v)$-$(vi)$ of Proposition \ref{principal}. To prove property $(vii)$, we remark that by  Proposition \ref{fondaprops} $(i)$ and $(iii)$ together with the hyperbolicity of $\arr \Omega$, the bundle $E^s_{i}$ is contracted by $F^\delta$ over a neighborhood of $\arr \Omega_i$, for every $i$.  Moreover, by properties $(i)$-$(iii)$-$(iv)$ of Proposition \ref{fondapropu}, the bundle  $E^u_{i}$ is close to $E^u|\arr \Omega_i $, and so expanded by $F^\delta$ on a neighborhood of $\arr \Omega_i$ (see Proposition \ref{prop:Eu-concrete}).
\end{proof}

\section{Proof of Proposition \ref{fondaprops}}\label{princis}
Let us recall that $(M_j)_{j=1}^q$ is a filtration adapted $(\arr
\Omega_j)_{i=1}^q$ (see \textsection \ref{filtration} for details). An example of such a filtration is depicted figure \ref{filtrationpic}.

\begin{figure}[h]
    \centering
        \includegraphics{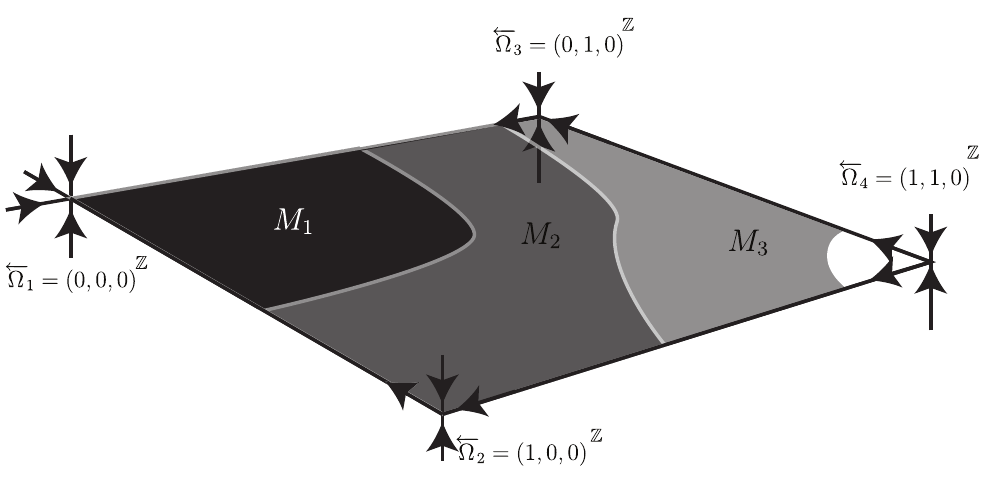}
    \caption{A filtration for the example given by $f\colon (x,y,z)\mapsto (x^2,y^2, 0)$.}
\label{filtrationpic}
\end{figure}

We are going to construct $(E_{i}^s)_i$ by (increasing) induction on $i$. Here it is the induction hypothesis at the step $i$:

For every $N_{i}\le 0$, there exist  neighborhoods $(V_j^i)_{j=1}^i$ of respectively $(W^s(\arr \Omega_j)\cap \arr f^{N_{i}}(M_{i}))_{j=1}^q$ in $ \arr f^{N_{i}}(M_{i})$, there are functions 
\[E_{j}^s\colon V_j^i\to G_N\]
which satisfy the following properties for every $j\le i$:
\begin{enumerate}[$(i)$]
\item for every ${\underline x}\in V^{i}_j\cap \arr f^{-1}(V^{i}_j)$ the following inclusion holds:
\[F^\delta(E_{j{\underline x}}^s)\subset E_{j\arr f({\underline x})}^s.\]
\item for every $k\ge j$, for every ${\underline x}\in V^{i}_k\cap V^{i}_j$ the following inclusion holds:
\[E_{k{\underline x}}^s\subset E_{j{\underline x}}^s.\]
\item $E_{j}^s$ is compact-open close to $E^s|W^s(\Omega_j)\cap V_j^i$, when $\delta$-is small. 
%
\item   $E_{j}^s$ is of constant dimension, $d_1$-continuous, and locally Lipschitz for the  metric $d_{\infty}$.
\end{enumerate}
We remark that  the step $i=q$ gives the statement of Proposition \ref{fondaprops} with $V_i:= \arr f^{-1} (V_i^q)\cap V_i^q$ for every $i$.

We recall that each $E_j^s$ depends on $\delta$. During several parameters will be fixed. 

The order is the following at the step $i$. First an arbitrary negative integer  $N_i$ is given. Then $\eta>0$ is chosen. Depending on $N_i$ and $\eta$, we will suppose $\delta$ small. The induction hypothesis is used with $\delta$ and $N_{i-1}$ chosen large in function of $N_i$ and $\eta$.

\paragraph{Step $i=1$}
Let $N_1\le 0$, and put $K_1:= \arr f^{N_1}(M_1)$. We notice that $W^s(\arr \Omega_1)$ is an open set of $\arr M_f$. Hence we put $V_1^1:=K_1$. Note that $\arr f(K_1)\subset K_1$.

Let $K_1\ni \underline x\mapsto E^\prime_{\underline x}$ be the restriction to $K_1$ of a smooth approximation 
of the continuous map $E^s|W^s(\arr \Omega_1)\colon W^s(\arr \Omega_1)\to G_N$ given by Corollary \ref{approximation general}.

Observe that $E^\prime$ is uniformly close to $E^s|K_1$. Moreover it is $d_1$-continuous and $d_{\infty}$-Lipschitz.
We recall that the Banach manifold $C^0(K_1, G_N)$ was defined in \textsection \ref{preli}.


For all $\eta>0$ and $\delta>0$, the following is well defined 
on the closed ball $\bar B_{C^0}(E^\prime,\eta)\subset C^0(K_1, G_N)$ with image in $C^0(K_1, G_N)$.
  \begin{displaymath}
    F^{\delta\#}:=\bar B_{C^0}(E^\prime,\eta)\ni P\mapsto \left[\underline x\mapsto
      F_{x_0}^{\delta^{-1}} (P_{\arr f(\underline x)})\right], \quad \text{with }x_0:=\pi_0(\underline x). 
  \end{displaymath}
By hyperbolicity, for $\delta$ small enough and   $E^\prime$ sufficiently close to $E^s$,
 there exists some $k\in\N$ such that ${F^{\delta\#}}^k$ is $\lambda$-contracting and sends the closed ball $\bar B_{C^0}(E^\prime,\eta)$ into itself. 

  Let $E_{1}^s$ be the unique fixed point of $F^{\delta\#}$ in $B_{C^0}(E^\prime,\eta)$. By definition,
  condition $(i)$ is satisfied. 

  Condition $(iii)$ follows from the fact that  $\eta$ can be taken small when $\delta $ is small. 
  
  It remains only to show that $(iv)$ holds. First let us recall that
  $E^s_{1}\in B_{C^0}(E^\prime,\eta)$ and $x_0 \in M\mapsto F^\delta_{x_0}\in
  L_N(\R)$ is of class $C^1$, and so, $K_1\ni \underline x \mapsto F^\delta_{x_0}$ is
  $d_\infty$-Lipschitz. Since $F^\delta$ is moreover invertible,  there exists $L_{\delta,k}$ such that for all
  $\underline x, \underline y\in K_1$ and $P\in B_{C^0}(E^\prime,\eta)$ it holds
  \begin{equation}
    \label{lip1}
    d\big({F^{\delta k}_{\underline x}}^{-1}(P_{ \arr f^k( \underline y)}),
    {F^{\delta k}_{\underline y}}^{-1} (P_{ \arr f^k( \underline y)})\big)\le L_{\delta,k} d_{\infty}( \underline x,  \underline y),
  \end{equation}
  where ${F^{\delta k}_{\underline x}}:= {F^{\delta}_{x_{k-1}}} \circ \cdots \circ  F^{\delta}_{x_0} $ and $x_i:=\pi_i(\underline x)$ the $i^{th}$ coordinate of $\underline x$.

  On the other hand, the map ${F^{\delta\#}}^k$ is pointwise $\lambda$-contracting:
  \begin{equation}
    \label{lip2}
    d\big({F^{\delta k}_{\underline x}}^{-1} (P_{ \arr f^k( \underline x)}),
    {F^{\delta k}_{\underline x}}^{-1} (P_{ \arr f^k( \underline y)})\big)\leq
    \lambda d(P_{ \arr f^k( \underline x)}, P_{ \arr f^k(
\underline       y)}). 
  \end{equation}

  Consequently, adding \eqref{lip1} and \eqref{lip2} we get 
  \begin{equation}
    \label{lip3}
    d\big({F^{\delta^{-k}}_{ \underline x}} (P_{ \arr f^k(
     \underline  x)}), {F^{\delta^{-k}}_{\underline y}} (P_{ \arr f^k( \underline y)}\big)
    \le L_{\delta,k} d_{\infty}( \underline x, \underline  y)+ \lambda d(P_{
      \arr f^k( \underline x)}, P_{ \arr f^k( \underline y)}).
  \end{equation}
For every $d_\infty$-Lipschitz distribution $P$ let us denote by $\Lambda(P)$ its Lipschitz constant. It holds for every $k$:
  \begin{equation}
    \label{lip4} 
    d(P_{ \arr f^k( \underline x)}, P_{ \arr f^k( \underline y)})\le \Lambda(P) d_{\infty} ( \underline x, \underline y).
  \end{equation} 

  Thus by  (\ref{lip3}) and (\ref{lip4}):
  \begin{displaymath}
    d\big({F^{\delta k}}^{-1}_{\underline x} (P_{ \arr f^k( \underline x)}), {F^{\delta k}}^{-1}_{\underline y}
    (P_{ \arr f^k( \underline y)}\big) \le (L_{\delta,k} + \lambda \Lambda(P))d_{\infty} ( \underline x, \underline  y) 
  \end{displaymath}

  Consequently the closed subset of $B_{C^0}(E^\prime,\eta)$ formed by sections with $d_\infty$-Lipschitz constant smaller or equal than
  any $\Lambda \ge (1-\lambda)^{-1}L_{\delta,k}$ is forward invariant under ${F^{\delta\#}}^k$.   We recall that $E^\prime$ is $d_{\infty}$-Lipschitz.
Hence if $\Lambda\ge \Lambda(E^\prime)$, this subset is  non empty (it contains $E^\prime$), thus  there exists a fixed point $d_{\infty}$-Lipschitz in $\bar B_{C^0}(E^\prime,\eta)$. By uniqueness, the fixed point $K_1\ni  \underline x \mapsto  E^s_{1 \underline x}\in G_N$ is $d_{\infty}$-Lipschitz.
  
\begin{flushright}
$\square$
\end{flushright}

\paragraph{Step $i-1\rightarrow i$}
Let $N_i$ be an arbitrary negative integer. Put:
\[K_{i}:=W^s(\arr \Omega_i)\cap \arr f^{N_i}( M_{i}).\]

Let us begin as in the step $i=1$.

We can extend $d_1$-continuously the section $E^s|K_{i}\colon K_{i}\to G_N$ to an open neighborhood of $K_i$. 
Let $\underline x\mapsto E^\prime_{\underline x}$ be a smooth approximation given by Corollary \ref{approximation general} of such a continuous extension. 

The section $E^\prime$ is well defined on a small neighborhood $Z_i$  of $K_{i}$ in $\arr f ^{N_i}(M_i)$ of the form:
\[Z_i :=  \arr f^{N_i}(M_i)\setminus int\, \arr f^{N_{i-1}+1}(M_{i-1}),\quad N_{i-1}\le 0\]
Indeed note that $K_{i}= \arr f^{N_i}(M_i) \setminus   \cup_{n\le 0} \arr f^n (M_{i-1})$, so $Z_i$ is close to $K_i$ whenever $-N_{i-1}$ is large enough (see fig. \ref{ConstructionZi}).
\begin{figure}[h]
    \centering
        \includegraphics{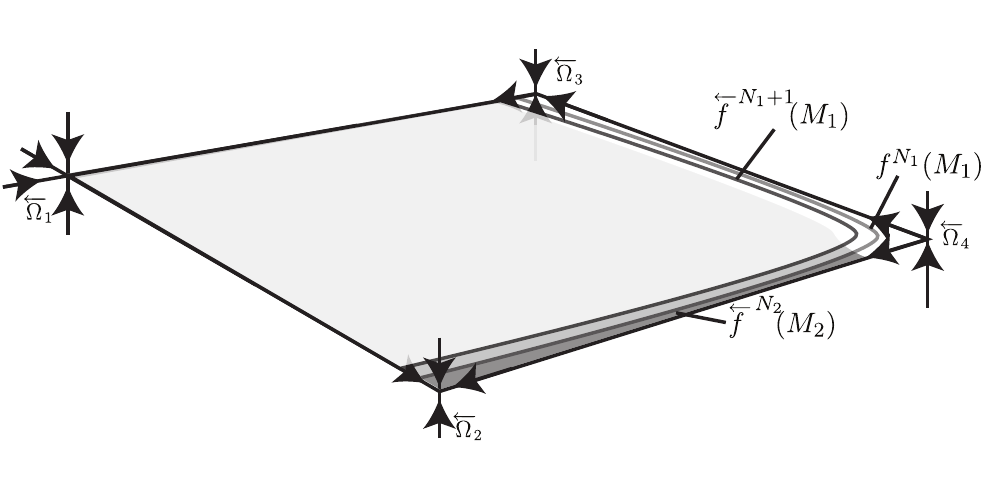}
    \caption{Construction of $Z_2= f^{N_2}(M_2)\setminus int\, f^{N_1+1}(M_1)$ in the example given by $f\colon (x,y,z)\mapsto (x^2,y^2, 0)$.}
\label{ConstructionZi}
\end{figure}

 Observe that $E^\prime|K_{i}$ is $C^0$-close to $E^s|K_i$, $d_1$-continuous and $d_{\infty}$-Lipschitz.

Hence,  for every $\eta$ small, for every $Z_i$ and $\delta$ small enough, for  $E^\prime$ sufficiently close to $E^s$, the following is well defined on the ball $\bar B_{C^0}(E^\prime,\eta)\subset C^0(Z_i, G_N)$ with image in $C^0( Z_i \cap \arr f^{-1}(Z_i), G_N)$:  
  \begin{displaymath}
    F^{\delta\#}:=\bar B_{C^0}(E^\prime|Z_i,\eta)\ni P\mapsto \left[ Z_i \cap \arr f^{-1}(Z_i)\ni \underline x\mapsto
      F_{x_0}^{\delta^{-1}} (P(\arr f(\underline x)))\right]. 
  \end{displaymath}

By hyperbolicity, for $\eta$ small and then  for $Z_i$ and $\delta$ small enough and $E^\prime$ sufficiently close to $E^s$, there exist some $k\in\N$,  such that the following map:
 \begin{displaymath}
    {F^{\delta\#}}^k:=\bar B_{C^0}(E^\prime|Z_i,\eta)\ni P\mapsto \left[\bigcap_{0\le l\le k} \arr f^{-l}(Z_i)\ni \underline x\mapsto
      {F_{x_0}^\delta}^{-k} (P(\arr f^k(\underline x)))\right]. 
  \end{displaymath}
is contracting and sends the closed ball $\bar B_{C^0}(E^\prime|Z_i,\eta)$ into $\bar B_{C^0}(E^\prime| \cap_{l\le k} \arr f^{-l}(Z_i),\eta)$.  

As the target space is not the same as the source space, we cannot conclude to the existence of a fixed point. We are going to extend the sections  in the image  of ${F^{\delta\#}}$ by a section constructed by the following lemma shown below:
\begin{lemma}\label{tildeE} 
There exist a sequence of negative integers $(N_j)_{j<i-1}$ and a section $\tilde E\in B_{C^0}(E^\prime|Z_i, \eta/2)$ which is $d_1$-continuous and $d_\infty$-Lipschitz such that for every $j<i$:
\begin{enumerate}[(1)] 
\item $Z_j :=   \arr f^{N_{j}} ( M_{j})\setminus int\, \arr f^{N_{j-1}+1} ( M_{j-1})\subset int\, V_j^{i-1}.$
\item $\forall \underline x \in Z_j \cap Z_i ,\quad  \tilde E(\underline x)\subset E_{j}^s(\underline x)$.
\end{enumerate}
\end{lemma}
\paragraph{Gluing $\tilde E$ to $    F^{\delta\#} \tilde E$ and definition of $E_{i}^s$}

We remark that $\tilde E$ and $    F^{\delta\#} \tilde E$ are well defined on:
\[V_i^i := \arr f^{N_i}(M_i)\setminus \arr f^{N_{i-1}}(M_{i-1}).\]

By Corollary \ref{partition1}, there exists a partition of the unity $(\rho, 1-\rho)\in \mathrm{Mor}_0^\infty (\arr M_f)^2$ subordinated to the cover $(\arr f^{N_{i-1}-1}(int\,M_{i-1}), \arr M_f \setminus \arr f^{N_{i-1}}(M_{i-1}))$. 

For $\underline x\in V_i^i$, let $p_{\underline x}$ and $p'_{\underline x}$ be the orthogonal
projections of $\mathbb R^{N}$ onto respectively $\tilde E(\underline x)$  and $    F^{\#} \tilde E(\underline x)$. Put
\begin{displaymath}
  E^0(\underline x):= \{(\rho(\underline x) p_{\underline x}+(1-\rho(\underline
  x)) p'_{\underline x})(u);\; u\in 
  \tilde E(\underline x)\}. 
\end{displaymath}

We notice that $ V_i^i\ni \underline x\mapsto E^0(\underline x)\in G_N$ is $d_1$-continuous and $d_\infty$-Lipschitz. 
Furthermore, for $\underline x\in V_i^i $ close to $\arr f^{N_{i-1}}( M_{i-1})$, the plane $E^0(\underline x)$ is equal to $\tilde E(\underline x)$ and for $\underline x\in V_i^i \setminus \arr f^{N_{i-1}-1}( int\, M_{i-1})$, the plane $E^0(\underline x)$ is equal to $F^{\delta\#} \tilde E_{\underline x}$.

We define
\begin{multline*}
\Pi_\eta:= \{P\in  C^0(V_i^i,G_N):\;  P=E^0\;\text{on}\; V_i^i \setminus \arr f^{N_{i-1}-1}( M_{i-1})\text{
and}\\
 \text{the restriction }P\text{ and } E^\prime\text{ to }\cap_{j=0}^k \arr f^j(V_i^i)\text{ are }\eta-C^0-\text{close}\}.
\end{multline*}
 We remark that the following map is continuous.
\begin{displaymath}
  F^{\delta\star}:=P\in \Pi_\eta\mapsto \left[V_i^i\ni  \underline x\mapsto
    \left\{
      \begin{array}{cl} 
        E^0(\underline x)&\text{if }\underline x\in V_i^i \setminus \arr f^{N_{i-1}-1}( M_{i-1})\\ 
        {F^\delta}_{\underline x}^{-1} (P(\arr f(x)))&\text{otherwise.}
      \end{array}\right.\right] 
\end{displaymath}

  As ${F^{\delta\#}}^k$ is contracting and sends $\bar B_{C^0} (E^\prime|Z_i ,\eta)$ into $\bar B_{C^0} (E^\prime|\cap_j \arr f^j (Z_i) ,\eta)$, the map $F^{\delta \star^k}$ is contracting and sends $\Pi_\eta$ into itself. 

Let $E_{i}^s$ be the fixed point of $F^{\delta\star}$. 

By definition, $E_{i} ^s$ satisfies property $(i)$. Similarly to the step $i=1$, the section $E_{i}^s$ satisfies Properties $(iii)$ and $(iv)$. However all the sections $(E_{j}^s)_{j\le i}$ need to be extended from $(V_j^{i-1})_{j\le  i}$ to $(V_j^{i})_{j\le  i}$ which remains to be constructed. 

\paragraph{Construction of $(V_j^{i})_{j\le  i}$ and extension of $(E_{j}^s)$}

For every $j <i $, we recall that for every   $\underline x\in Z_j\cap Z_i$, the plane $\tilde E_{\underline x}$ is included in  
$E^s_{j\underline x}$. By induction hypothesis $(i)$ and since $V_i^i= Z_i \cap \arr f^{-1}(Z_i)$, for every   $\underline x\in Z_j\cap V^i_i$, the plane $F^{\delta \#}\tilde E_{\underline x}$ is included in $E^s_{j\underline x}$. Hence, for every   $\underline x\in Z_j\cap V^i_i$, the plane $ E^0_{\underline x}$ is included in $E^s_{j\underline x}$. Again by induction hypothesis $(i)$, since the fixed point $E^s_{i\underline x}$ of $F^{ \star}$ is obtained by iterating it, for every   $\underline x\in Z_j\cap V^i_i$, the plane $E^s_{i\underline x}$  is included in $E^s_{j\underline x}$. Put 
\[\tilde V_j:= V_j^{i-1} \cap \cup_{l=j}^{i-1} Z_l\]
By induction hypothesis $(ii)$, for every   $\underline x\in \tilde V_j\cap V^i_i$, the plane  $E^s_{i \underline x}$ is included in $E^s_{j\underline x}$.
Using that $Z_j = \arr f^{N_j}(M_j) \setminus int\, \arr f^{N_{j-1} +1}(M_{j-1})$, we remark that: 
\[\tilde V_j  \supset  V_j^{i-1}  \cap \arr f^{N_{i-1}}(M_{i-1})\setminus int \,  \arr f^{N_{j-1} +1}(M_{j-1})= V_j^{i-1} \setminus int \,  \arr f^{N_{j-1} +1}(M_{j-1}).\]
Hence $\tilde V_j$ is a neighborhood of $W^s(\arr \Omega_j)\cap \arr f^{N_{i-1}}(M_{i-1})$ since $W^s(\arr \Omega_j)$ does not intersect  $\arr f^{N_{j-1}+1}(M_{j-1})$.  Put 
\[ V_j^i:= \cup_{n\ge 0} (\arr f|V_i^i )^{-n}(\tilde V_j)\quad \text{where}\quad V_i ^i =  \arr f^{N_i}(M_i)\setminus \arr f^{N_{i-1}}(M_{i-1})\]
\begin{lemma}
For every $j\le i$, the set $V_j^i$ is a neighborhood of $W^s(\Omega_j)\cap f^{N_i}(M_i)$ in $f^{N_i}(M_i)$.\end{lemma}
\begin{proof}
As the case $i=j$ is obvious, we suppose $j<i$. For every $\underline x \in f^{N_i}(M_i)\cap W^s(\arr \Omega_j)$ there exists $n $ such that $\arr f^n(\underline x)\in  int(f^{N_{i-1}}(M_{i-1}))$. Consequently, for every $\underline  y\in \arr f^{N_i}(M_i)$ nearby $\underline  x$, there exists $m$ such that 
$\arr f^m(\underline y)\in \arr f^{N_{i-1}}(M_{i-1})$.
 Let us consider such an $m$ minimal. 
 Since $\tilde V_j $ is a neighborhood of $W^s(\arr \Omega_j)\cap f^{N_{i-1}}(M_{i-1})$ in $f^{N_{i-1}}(M_{i-1})$, the point $\arr f^m(\underline y)$ belongs to $\tilde V_j$ for $\underline y$ sufficiently close to $\underline x$ . 
Also for every $k<m$ the point $\arr f^k(\underline y)$ belongs to the complement of $f^{N_{i-1}}(M_{i-1})$. On the other hand,  $\arr f^k(\underline y)$ belongs to $ f^{N_i}(M_i)$ for every $k$. Thus for every $k<m$, the point $\arr f^k(\underline y)$ belongs to $V_i ^i=\arr f^{N_i}(M_i)\setminus \arr f^{N_{i-1}}(M_{i-1})$. As $\arr f^m(\underline y)$ belongs to $\tilde V_j$, it follows that  $\underline y$ belongs to $V_j^i$.\end{proof}

We extend $E_{j}^s$ on $V_j^i$ by:
\[\forall x\in V_j^i, \forall n\ge 0 \text{ minimal such that } \arr f^n(\underline x) \in \tilde V_j,\;  E_{j\underline x} ^s = {F_{\underline x}^\delta}^{-n}(E_{j  \arr f^n(\underline x)}^s).\]
Induction hypotheses $(i)$, $(iii)$ and $(iv)$ for $j<i$ imply properties $(i)$, $(iii)$  and $(iv)$ for $E_{j}^s$ on $V_j ^{i}$. 


Let us check property $(ii)$. As this property is invariant by pull back, property $(i)$ implies that property $(ii)$ holds for every $k\le j$ both less than $i$. Let $\underline x\in V_i ^i\cap V_j ^i$. Let $n$ be such that $\underline x$ belongs to $\arr f^{-n}(\tilde V_j )\cap V_i^i$. We recall that $V_i ^i=\arr f^{N_i}(M_i)\setminus \arr f^{N_{i-1}}(M_{i-1})$, hence $\arr f^n (\underline x)$ belongs to $\arr f^{N_i+n} (M_i)\subset \arr f^{N_i}(M_i)$. Also $\arr f^n (\underline x)$ belongs to $\tilde V_j \subset V_j ^{i-1}\subset f^{N_{i-1}}(M_{i-1}^c)$. Thus $\arr f^n(\underline x)\in \tilde V_j \cap V_i^i$. Consequently, Property $(ii)$ holds at $\arr f^n (\underline x)$. 
By pull back invariance and property $(i)$, Property $(ii)$ holds at $\underline x$.

\paragraph{Proof of Lemma \ref{tildeE}}
We are going to project $E^\prime$ onto each $E^s_{j}$, $j\le i$. The following is a consequence of the Lambda-lemma and the strong transversality condition.

\begin{claim}\label{claim71}
For $Z_i$ small enough (that is $-N_{i-1}$ large enough), for $\delta$ small enough, for every $j<i$, every $\underline x\in W^s(\arr \Omega_j)\cap Z_i$, if $q_j(\underline x)$ denotes the orthogonal projection of $\mathbb R^N$ onto $E^s_{\underline x}$, the distance between $q_j(\underline x)(E^\prime_{\underline x})$ and $E^\prime_{\underline x}$ is less than $\eta/4i$.\end{claim}
\begin{proof} 
By the strong transversality condition, on $W^u_\epsilon(\arr \Omega_i)\setminus \arr \Omega_i$ the stable direction $E^s$ is transverse to  $TW^u_\epsilon(\arr \Omega_i)$. This is true in particular on $W^u_\epsilon(\arr \Omega_i)\cap W^s(\arr \Omega_j)$. By hyperbolicity and  the strong transversality condition, 
for every $(\underline x_n)_n$ in $W^u_\epsilon(\arr \Omega_i)\cap W^s(\arr \Omega_j)$ approaching $\underline x \in \Omega_i$, 
every accumulation plane $P$ of $(E^s_{\underline x_n})_n$ contains the plane $E^s_{\underline x}$. This implies
that for every $(\underline x_n)_n$ in $W^s(\arr \Omega_j)$ approaching $\underline x \in K_i= \arr f ^{N_i}(M_i)\cap W^s(\arr \Omega_i)$, 
every accumulation plane $P$ of $(E^s_{\underline x_n})_n$ contains the plane $E^s_{\underline x}$.
The claim follows since $E^\prime$ is close to $E^s|K_i$ for $\delta$ small and $Z_i$ small.
\end{proof}

We will perform orthogonal projections of $E^s_{i}$ on compact subsets of each $Z_i \cap W^s(\arr \Omega_j)$.

 Let us implement these compact subsets.

First let us notice that $cl(Z_i \setminus \arr f^{-1}(Z_i))$ is a compact subset of $\cup_{j< i} W^s(\Omega_j)$:
\[Z_i \setminus \arr f^{-1}(Z_i)=
\arr f^{N_i}(M_i)\setminus int\, \arr f^{N_{i-1}+1}(M_{i-1}) \setminus \big(
\arr f^{N_i-1}(M_i)\setminus int\, \arr f^{N_{i-1}}(M_{i-1})
\big).\]

 As $\arr f^{N_i}(M_i)$ is included in $\arr f^{N_i-1}(M_i)$ it comes:
 \[Z_i \setminus \arr f^{-1}(Z_i)= 
   \arr f^{N_i}(M_i)\cap int\, \arr f^{N_{i-1}}(M_{i-1})\setminus int\, \arr f^{N_{i-1}+1}(M_{i-1})\subset \cup_{j< i } W^s(\arr \Omega_j).\]

Let $(V_k^{i-1})_{k\le i-1}$ be the neighborhoods given by the induction hypothesis at step $i-1$ for the integer $N_{i-1}$ defined above. 

By decreasing induction we construct $(N_k)_{k=1}^{i-2}\in \mathbb Z^{-}$ such that, the following holds.
\begin{claim}\label{claim72}
For every $k\le i-1$, the set $Z_{k}:= \arr f^{N_k}(M_{k})\setminus \arr f^{N_{k-1}+1}(M_{k-1})$ has its closure included in the interior of $V_{k}^{i-1}$. Moreover, 
for $\delta$ small enough, the distance  between the orthogonal projection $p_{k}$ onto $E^s_{k}$ satisfies:
\begin{equation}\label{distproj} 
\|p_{k}(\underline x)(E^\prime_{\underline x}) - E^\prime_{\underline x}\|\le \eta/3i, \quad \forall \underline x\in Z_{k}.\end{equation},
\end{claim}\label{claimdist}
\begin{proof} For $k\le i-1$, suppose $N_k$ constructed. Then for $-N_{k-1}$ large, the set $Z_k$ is close to the compact set $W^s (\arr \Omega_k)\cap  f^{N_k}(M_k)$, and so it is included in $V_k$. Moreover, by remark \ref{defproche} and Claim \ref{claim71} for $-N_{k-1}$ large and $\delta$ small, inequality (\ref{distproj}) holds.\end{proof}

Let $\hat Z_k$ be a neighborhood of $Z_k $ in $V_k^{i-1}$ such that for all $\underline x\in \hat Z_k $, 
\[\|p_k (\underline x) (E^\prime _{\underline x}) -E^\prime _{\underline x}\|\le \frac{\eta}{2 i}.\]
By corollary \ref{partition1}, there exists a dump function
$\rho\in Mor_0^\infty (\arr f^{N_{i-1}}(M_{i-1}), [0,1])$ equal to $1$ on $Z_k$ and to $0$ on $\hat Z_k ^c$. We construct $(P_{\underline x}^j)_{j<i}$ by induction. Put $P^0= E^\prime$, and for $j \in [1, i-1]$ put
\[P_{\underline x} = \{\rho_j (\underline x) \cdot q_j (\underline x) (u) +(1-\rho_j(\underline x)) \cdot u : \; u \in P_{\underline x}^{j-1}\}\]
Let $\hat E = P_{\underline x} ^{i-1}$. By induction hypothesis $(ii)$, for every $\underline x \in Z_i \cap Z_j$, it holds $\tilde E_{\underline x}\subset E^s_{j\underline x}$.

%
%
%
%
%

By definition of $\hat Z_k $, the section $ \tilde E$ is in $B_{C^0}(E^\prime|Z_i,\eta/2)$ and is $d_\infty$-Lipschitz. 

\section{Proof Proposition \ref{fondapropu}}\label{princiu}
The proof of Proposition \ref{fondapropu} is done by decreasing induction on $q'\in [1,q]$. We recall that 
Proposition \ref{fondaprops} constructed sections $(E_{j}^s)_j$ on neighborhoods $(V_j)_j$ of respectively $(W^s_j (\arr \Omega_j))_j$, which satisfy properties $(i)$-$(ii)$-$(iii)$ and $(iv)$. Here is the induction hypothesis.

For every $q'\le q$, there exist $K>0$ and an open cover
  $(W_i)_{i=q'}^{q}$ of $\cup_{j\ge q'} W^s (\arr \Omega_j)$, where each $W_i$ is a neighborhood
  of $\arr \Omega_i$ included in $V_i^i$ and  such that for every $\delta>0$, there exists a function $E_i^u\in C^0(W_i, G_N)$ satisfying the following properties for $i\in [q',q]$:
  \begin{enumerate}[$(i)$]
  \item For every $\underline x\in W_i\cap \arr f^{-1}(W_i)$, the map $F^\delta$ sends $E_{i\underline x}^u$ into $E_{i\arr f(\underline x)}^u$.
  \item For every $j\ge i$, for every $\underline x\in W_i\cap\arr f^{-1} (W_j)$ the following inclusion holds:
\[E_{i{\underline x}}^u\subset E_{j\arr f({\underline x})}^u.\]
  \item $E^s_{i\underline x}\oplus E^u_{i\underline x}=\R^N$, for every $\underbar x\in W_i$;
  the angle between $E_{i}^s$ and $E_{i}^u$ is bounded from below by $K^{-1}$.
  \item The subbundle $E_{i}^u$ is of constant dimension, $d_1$-continuous and $d_\infty$-Lipschitz.
  \item For  any $\underline x\in W_i$, it holds
    \begin{displaymath}
      \|{F^\delta}(v^u)\|\ge \|v^u\|/K, \quad\forall v^u\in
      E^u_{i\underline x}. 
    \end{displaymath}
    \item It holds $ cl(\arr f^{-1}(\cup_{q'}^qW_i) )\subset \cup_{q'}^q W_i$. Moreover, for every $j\ge i$, if $\underline x\in W_{j}$ then $\arr f(\underline x)\notin \cup_{k>j} W_k$ and $\cap_{n\in \mathbb Z} \arr f^n(W_{j})=\arr \Omega_{j}$.
  \end{enumerate}
  
We continue to denote by $(M_j)_{j=1}^q$ a filtration adapted to $(\arr
\Omega_j)_{i=1}^q$ (see \textsection  \ref{filtration} for details and fig. \ref{filtrationpic}).   

At each step $q'$ of the induction we will work with a small $\eta$ and we will suppose an integer $-N_{q'}$ large and $\delta$ small both depending on $\eta$. 
\paragraph{Step $q'= q$}
The subset $\arr \Omega_q=W^s(\arr \Omega_q)$ is compact. Moreover there exists an arbitrarily small compact neighborhood $W_q$ of $\arr \Omega_q$ which satisfies $(vi)$ for $i=q$.
Indeed,  consider $W_q$ of the form $\arr M_f \setminus \arr f^{-N}( M_{q-1})$. Hence we can suppose that $W_q$ is included in $V_q$. 

Let $\eta>0$ be small, in particular smaller than the angle between $E^s|\Omega(\arr f)$ and $E^u|\Omega(\arr f)$.

Let $\underline x\mapsto E^\prime_{\underline x}$ be the restriction to $W_q$ of a smooth approximation of a continuous extension of the continuous map $E^u|\arr \Omega_q\colon \arr \Omega_q\to G_N$ given by Corollary \ref{approximation general}. This means that 
on the one hand,   $E^\prime$ is $d_1$-continuous and $d_{\infty}$-Lipschitz, and that for every $\epsilon$ small, if $W_q$ is sufficiently small then for every $\underline x\in W_q$ there exists $\underline y\in \arr \Omega_q$ $\epsilon$-close to $\underline x$ such that the distance between $E^\prime_{\underline x}$ and $E^u_{\underline y}$ is $\eta$ small.  

By hyperbolicity of $\arr \Omega_q$, the angle between $E^s_{\underline y}$ and $E^u_{\underline y}$ is uniformly bounded from below on $\underline y\in \arr \Omega_q$  and $Tf|E^u_{\underline y}$ is bijective. By property $(iii)$ of Proposition \ref{fondaprops} and remark \ref{defdistancesurensemblesdifferents}, there exists $K$ large such that for every $\eta>0$ small, for every $W_q$  sufficiently small, for all $\delta\ge 0$ small, and for every  $\underline y\in W_q$ the following holds:
\begin{itemize}
\item[$(a)$]  the angle between $E^s_{q \underline y}$ and $E^\prime_{\underline y}$ is greater than $K^{-1}$, 
\item[$(b)$] for every plane $P$ making an angle  with  $E^\prime_{\underline y}$ smaller than $\eta$, it holds:
\[\forall u\in P, \quad \|F^\delta(u)\|\ge \|u\|/K.\]
\end{itemize}
Indeed, for $\underline x\in \arr \Omega_q$, every vector $u$ in $E^u$ is expanded by $F^0$. 

We can now proceed as in the step $i=1$ of the proof of Proposition \ref{fondaprops}.
 
Since for every $\delta>0$, the map $F^\delta$ is bijective,  
 the following is well defined   
  \begin{displaymath}
    F_\#:=\bar B_{C^0}(E^\prime,\eta )\ni P\mapsto \left[\underline x\mapsto
      F_{x_0}^{\delta} (P_{\arr f^{-1}(\underline x)})\right]\in C^0(W_q, G_N). 
  \end{displaymath}
Moreover, for $\delta$, $W_q$ small enough and $E^\prime$ close enough to $E^u$, 
 there exists some $k\in\N$ such that $F_\#^k$ is contracting and sends the closed ball $\bar B_{C^0}(E^\prime,\eta)$ into itself.  

  Let $E_{q}^s$ be the unique fixed point of $F_\#$ in $B_{C^0}(E^\prime,\eta)$. In this way,
  condition $(i)$ is clearly satisfied. 

  Properties $(iii)$ and $(v)$ follow from respectively Properties $(a)$ and $(b)$ above. Property $(ii)$ is empty.

  To prove property $(iv)$, we proceed as in the proof of Proposition \ref{fondaprops}, step $i=1$. 
\begin{flushright}
$\square$
\end{flushright}

 \paragraph{Step $q'+1\rightarrow q'$.} 
Let us suppose the neighborhoods $(W_i)_{i=q'+1}^q$ constructed so that 
\begin{itemize}
\item property $(vi)$ holds,
\item $W_i$ is a neighborhood of $\arr \Omega_i$,
\end{itemize}

Let us proceed again as in the proof of Proposition \ref{fondaprops} step $i-1\rightarrow i$.

We remark that $C_{q'}:=W^s(\arr \Omega_{q'})\setminus \arr f^{-2}(O_{q'})$ is compact, with $O_{q'} :=\cup_{i=q'+1}^q W_i$. Moreover for every $N_{q'}\le 0$, the following is a compact set containing $C_{q'}$:
\[Y_{q'}:=\arr f^{N_{q'}}(M_{q'-1}^c)\setminus \arr f^{-2}(O_{q'}),\]
Moreover, when  $-N_{q'}$ is large, $Y_{q'}$ is close to $C_{q'}$ for the Hausdorff metric. By $Y_{q'}$ \emph{small} we mean $-N_{q'}$ large.
 
First, we assume $-N_{q'}$ large enough so that the set $Y_{q'}$ is included in $V_{q'}$. 

By strong transversality and property $(iii)$ of Proposition \ref{fondaprops}, for every $\eta>0$, there exists $K$ large such that for every $\delta$ and  $Y_{q'}$ small, it holds:\begin{equation}\label{transversilty3}
 \forall \underline x\in Y_{q'},\quad \forall u\in \R^N\setminus \{0\}: \quad |\angle(u,E_{q' \underline x}^s)|>\eta\Rightarrow  \|F^\delta(u)\|\ge \|u\|/K.\end{equation}
Indeed, if $\underline x\in C_{q'}$, a unit vector $u$ making an angle at least $\eta$ with $E_{\underline x}^s$ has
its image by $F^0$ not in $E_{\arr f(\underline x)}^s$, by the strong tranversality condition. Hence the norm if its image is bounded from below by a certain $1/2K$. Consequently for $\delta$ and $Y_{q'}$ small inequality (\ref{transversilty3}) holds.

For $\eta>0$, let $U_\eta$ be the closed subset of $C^0(Y_{q'}, G_N)$ made by sections $P$ such that for every $\underline x\in Y_{ q'}$ the angle between $P_{\underline x}$ and  $E_{q' \underline x}^s$ is at least $\eta$.

For all $\eta$, $\delta$ and $Y_{q'}$, the following is well defined with image in $C^0(Y_{q'} \cap \arr f(Y_{q'}), G_N)$:  
  \begin{displaymath}
    F_\#:=U_\eta \ni P\mapsto \left[\underline x\in Y_{q'} \cap \arr f(Y_{q'})\mapsto
      F_{x_0}^{\delta} (P_{\arr f^{-1}(\underline x)})\right]. 
  \end{displaymath}

Similarly, for every $k\ge 0$,  for all $\eta$, $\delta$ and $Y_{q'}$, the following is well defined with image in $C^0(\cap^k_{i=0} \arr f^i(Y_{q'}), G_N)$:  
  \begin{displaymath}
    F_\#^k:=U_\eta \ni P\mapsto \left[\underline x\in \cap^k_{i=0} \arr f^i(Y_{q'})\mapsto
      {F_{x_0}^{\delta}}^k (P_{\arr f^{-k}(\underline x)})\right]. 
  \end{displaymath}
We remark that $\cap^k_{i=0} \arr f^i(Y_{q'})=f^{N_{q'}}( M_{q'-1}^c)\setminus \arr f^{-2+k}(O_{q'})$ is close to  $\arr\Omega_{q'}$, when  $k$ is large and $Y_{q'}$ is small (that is $-N_{q'}$ large). 
We assume  $\eta>0$ smaller than the angle between $E^s|\Omega_{q'}$ and $E^u|\Omega_{q'}$.
Hence by hyperbolicity, for $Y_{q'}$ and $\delta$ sufficiently small, there exists $k$ such that $F_\#^k$ is contracting for the $C^0$-metric. Note that $k$ does not depend on $\eta$.
Moreover, by hyperbolicity, if $k$ is large enough and $\delta$ small enough, for every $\underline x  \in \arr f^{k}(C_{q'})$, for every $P_{\underline x}$ $\eta$-close to $E_{q'\underline x}^s$, the plane ${F^\delta_{\underline x}}^{-k}(P_{\underline x})$ is $\eta$-close to  $E^s_{\arr f^{-k}(\underline x)}$. Hence every $\underline x\in C_p$, every $P$ making an angle greater than $\eta$ with $E^s_{\underline x}$, the plane ${F_{\underline x}^\delta }^k(P_{\underline x})$ makes an angle greater than $\eta$ with $E_{q'\arr f^{k}(\underline x)}^s$. 

Consequently, for $-N_{q'}$ large enough, $F_\#^k $ takes its values in the subspace of $C^0(\cap _{i=0}^k f^i(Y_{q'}), G_N)$ formed by sections $P$ such that for every $\underline x \in \cap _{i=0}^k f^i(Y_{q'})$, $P_{\underline x}$ makes an angle with $E^s_{q'\underline x}$ greater than $\eta>0$. 
 
However, the target space of $F_\#$ is not the same as the source space. So we cannot conclude to a fixed point. We are going to complement the sections in the image of $F_\#$ by sections obtained by the following lemma shown below:
\begin{lemma}\label{tildeEu} 
For $\delta$ and $Y_{q'}$ small enough, there exists a $d_1$-continuous and $d_\infty$-Lipschitz section $\tilde E\in U_\eta$ such that for every $j>q'$:

\[\underline x\in Y_{q'}\cap \arr f(W_j),\quad  \tilde E_{\underline x} \subset F^{\delta}( E_{ j \arr f^{-1}(\underline x)}^u).\]

\end{lemma}

\paragraph{Gluing $\tilde E$ to $    F_\# \tilde E$ and definition of $E_{i}^u$}

We remark that $\tilde E$ and $F_\# \tilde E$ are well defined on:
\[
W_{q'}:= \arr f^{N_{q'}}(int \, M_{q'-1}^c)\setminus cl(\arr f^{-1}(O_{q'}))\subset Y_{q'} \cap \arr f(Y_{q'}).
\]
We remark that induction hypothesis $(vi)$ is satisfied.

By Corollary \ref{partition1}, there exists a partition of the unity $(\rho, 1-\rho)\in Mor_0^\infty (\arr M_f)^2$ subordinated to the cover $( O_{q'}, int\, \arr f^{-1} (O_{q'})^c)$. 

 Let $p_{\underline x}$ and $p'_{\underline x}$ be the orthogonal
projections of $\mathbb R^{N}$ onto respectively $\tilde E$  and $    F_\# \tilde E$. For $\underline x\in W_{q'}$, put
\begin{displaymath}
  E^0_{\underline x}:= \{(\rho(\underline x) p_{\underline x}+(1-\rho(\underline
  x)) p'_{\underline x})(u);\; u\in 
  \tilde E_{\underline x}\}. 
\end{displaymath}

We notice that $ \underline x\in W_{q'}\mapsto E^0_{\underline x}\in G_N$ is $d_1$-continuous and $d_\infty$-Lipschitz. 
Furthermore, for $\underline x\in W_{q'} $ close to $\arr f^{-1}( O_{q'})$, the plane $E^0_{\underline x}$ is equal to $\tilde E_{\underline x}$ and for $\underline x\in W_{q'} \setminus O_{q'}$, the plane $E^0_{\underline x}$ is equal to $F_\# \tilde E_{\underline x}$.

We put
\begin{multline*}
\Pi_\eta:= \{P\in  C^0(W_{q'},G_N):\;  P=E^0\;\text{on}\; W_{q'} \setminus O_{q'}\text{
and}\\
P\text{ and } E_{q'}^s\text{ makes an angle greater than }\eta\text{ on }W_{q'} \setminus \arr f^{-2+k}(O_{q'})\}.
\end{multline*}

We put:
\begin{displaymath}
  F^\delta_\star:=P\in \Pi_\eta\mapsto \left[\underline x\in W_{q'} \mapsto
    \left\{
      \begin{array}{cl} 
        E^0_{\underline x}&\text{if }\underline x\in W_{q'}\setminus \arr f^{-1}(O_{q'})\\ 
        {F^\delta}_{\underline x} (P_{\arr f^{-1}(x)})&\text{otherwise.}
      \end{array}\right.\right] 
\end{displaymath}

We notice that the map $F^\delta_\star$ takes its values in $C^0(W_{q'}, G_N)$. Moreover, from the properties of $F^\delta_\#$, the map $F^\delta_\star$ is $\lambda$-contracting and takes its values in $\Pi_\eta$.  


Let $E_{q'}^u$ be the fixed point of ${F^\delta_\star}$. 

By definition, $E_{q'}^u$ satisfies property $(i)$. Similarly to the step $i=1$, the section $E_ i^s$ satisfies Property $(iv)$.

Also, Property $(iii)$ is satisfied for every $P\in F_\star ^{\delta ^k}\Pi_\eta$, if $-N_{q'}$ is large enough and 
$\delta$ small enough. Hence it holds for $E^u_{q'}$. Likewise by (\ref{transversilty3}), if $-N_{q'}$ is large enough and 
$\delta$ small enough, property $(v)$ holds for $E^u_{q'}$. This gives a bound on $K$. Such a bound at this step  does not depend on $\delta$ small enough.

Let us check Property $(ii)$. We only need to check that for $j>q'$, for $\underline x\in \arr f^{-1}(W_{q'})\cap W_j$ it holds:
\[ E_{q' \arr f(\underline x)}^u \subset F^{\delta}( E_{j \underline x}^u).\]


That is for every $\underline x\in W_{q'}\cap \arr f(W_j)$ it holds:
\[ E_{q' \underline x}^u \subset F^{\delta}( E_{j \arr f^{-1}(\underline x)}^u).\]

Let $\underline x\in W_{q'}\cap \arr f(W_i)=  \arr f^{N_{q'}}(int \, M_{q'-1}^c)\setminus cl(\arr f^{-1}(O_{q'}) )\cap \arr f( W_j)$.  In particular $\underline x$ belongs to $W_{q'} \cap \arr f(O_{q'})\setminus \arr f^{-1}(O_{q'})$. 

If $\underline x$ belongs to $W_{q'} \cap O_{q'}\setminus \arr f^{-1}(O_{q'})$, then $E_{q'\underline x}^u$ is a linear sum of vectors included in $\tilde E_{\underline x}$ and $F^\delta( \tilde E_{ \arr f^{-1} (\underline x)})$. We recall that  $\tilde E_{\underline x}$ is included in $F^\delta(E^u_{i\arr f^{-1}(\underline x)})$ by Lemma \ref{tildeEu}. Also $\tilde E_{\arr f^{-1}(\underline x)}$ is included in $F^\delta(E^u_{j\arr f^{-2}(\underline x)})$ with $j\ge i$ such that $\arr f^{-2}(\underline x)\in W_j$. By $(ii)$, 
$F^\delta(\tilde E_{\arr f^{-1}(\underline x)})$ is included in $F^\delta(E^u_{i\arr f^{-1}(\underline x)})$. Hence $E^u_{q' \underline x}$ is included in $F^\delta(E^u_{i\arr f^{-1}(\underline x)})$.

If $\underline x$ belongs to $W_{q'} \cap \arr f (O_{q'})\setminus O_{q'}$, then $E_{q'\underline x}^u$ is a linear sum of vectors included in  $F^\delta( \tilde E_{ \arr f^{-1} (\underline x)})$ and ${F^\delta}^2( \tilde E_{ \arr f^{-2} (\underline x)})$.
As in the previous case, $F^\delta(\tilde E_{\arr f^{-1}(\underline x)})$ is included in $F^\delta(E^u_{i\arr f^{-1}(\underline x)})$. 
Similarly,   $\tilde E_{ \arr f^{-2} (\underline x)}$ is included in $F^\delta(E^u_{j\arr f^{-3}(\underline x)})$ with $j\ge i$ such that $\arr f^{-3}(\underline x)\in W_j$. By $(ii)$, 
${F^\delta}^2(\tilde E_{\arr f^{-2}(\underline x)})$ is included in ${F^\delta}^3(E^u_{j\arr f^{-3}(\underline x)})\subset F^\delta(E^u_{i\arr f^{-1}(\underline x)})$. Hence $E^u_{q' \underline x}$ is included in $F^\delta(E^u_{i\arr f^{-1}(\underline x)})$.

%

\paragraph{Proof of Lemma \ref{tildeEu}}
We want to construct $ \tilde E\in U_\eta$  which is $d_1$-continuous and $d_\infty$-Lipschitz and such that for every $j>q'$:
\[\underline x\in Y_{q'}\cap \arr f(W_j),\quad  E_{q' \underline x}^u \subset F^{\delta}( E_{j \arr f^{-1}(\underline x)}^u).\]

By property $(v)$, the section 
\[\underline x \in \arr f( W_{j})\mapsto {F^\delta}( E_{j\arr f^{-1}(\underline x)}^u)\] 
is $d_1$-continuous and $d_\infty$-Lipschitz. 


For every $j>q'$, let $W_j'$ be an open neighborhood of $\arr \Omega_j$ with closure in $W_{j}$ and such that
$\cup_{j>q''} W_j '$ contains $\cup _{j> q''} W^s(\arr \Omega_j)$ for every $q''\ge q'$ and $(vi)$ is satisfied.  

For every $j>q'$, let $\rho_j$ be a dump function equal to 1 on $\arr f( W'_j)\cup W'_j$ and with support in $\arr f( W_{j})\cup W_j$. We remark that $(\rho_j, 1-\rho_j)$ is a partition of the unity subordinate to the cover $(\arr f( W_{j})\cup W_j, \arr f( {W'_j})^c\cap {W'_j}^c  )$.

Let $p_j(\underline x)$ be the projection of $\R^N$ onto $F^\delta(E_{j\arr f^{-1}(\underline x)}^u)$ parallelly to $E_{j\underline x}^s$.

Let $\tilde E_{q'}$ be the orthogonal of $E_{q'}^s|Y_{q'}$. 

By Claim \ref{claim71}, for every $j$, for $Y_ {q'}$ and $\delta$-small enough, the plane $\tilde E_{q'}$ makes an angle greater than a certain $\eta>0$ with $E^s_{q' \delta}(\underline x)$. Hence its projection by $p_j$ remains of constant dimension and so is $d_1$-continuous and $d_\infty$-Lipschitz. 

We now construct inductively $(\tilde E_j)_{j\ge q'}$. 

Let $j\ge q'$ and let  us suppose the section $\tilde E_{j}\in U_\eta$, $d_1$-continuous and $d_\infty$-Lipschitz, constructed so that for every $i\in (q',j)$:
\begin{itemize}
\item for every $\underline x\in \arr f( W_i')$, the plane $\tilde E_{j\underline x} $ is included in $F^\delta(E_{i \arr f^{-1}(\underline x)}^u)$. 
\item for every $\underline x\in W_i'$, the plane $\tilde E_{j\underline x} $ is included in $E_{i \underline x}^u$. 
\end{itemize}
 Put:
\[\tilde E_{j+1 \underline x}:=
 \big\{\rho_{j +1}(\underline x)\cdot p_{j+1}(\underline x)(u)+(1-\rho_{j+1}(\underline x))\cdot (u):\; u\in \tilde E_{j \underline x}\big\}.\]

We define $\tilde E= \tilde E_q$.

We remark that by $(ii)$:
 \[\forall \underline x\in Y_{q'}\cap \arr f( W'_{j+1}),\quad  E_{q' \underline x}^u \subset F^{\delta}( E_{j \arr f^{-1}(\underline x)}^u).\]

 Hence by replacing $(W_j)_j$ by $(W_j')_j $, Lemma \ref{tildeEu} is proved. 
\bibliographystyle{amsalpha} \bibliography{references}
\end{document}